\documentclass[a4paper,11pt]{article}
\usepackage{latexsym}
\usepackage{amssymb,amsmath,array}
\usepackage{graphics}

\usepackage{xparse,todonotes}
\usepackage{url}

\usepackage{color}

\newcommand{\f}{\ensuremath{\varphi}}
\newcommand{\p}{\ensuremath{\psi}}
\newcommand{\1}{\ensuremath{\overline{1}}}
\newcommand{\0}{\ensuremath{\overline{0}}}

\DeclareDocumentCommand\addreference{g}{\todo[color = blue!30]{Add reference\IfNoValueF{#1}{: #1}}}
\DeclareDocumentCommand\checkthis{g}{\todo[color = red!50]{Check this\IfNoValueF{#1}{: #1}}}
\DeclareDocumentCommand\fixthis{g}{\todo[color = orange!50]{Fix this\IfNoValueF{#1}{: #1}}}
\DeclareDocumentCommand\expand{g}{\todo[color = green!50]{Expand\IfNoValueF{#1}{: #1}}}

\newtheorem{definition}{Definition}[section]
\newtheorem{theorem}[definition]{Theorem}
\newtheorem{lemma}[definition]{Lemma}
\newtheorem{proposition}[definition]{Proposition}
\newtheorem{remark}[definition]{Remark}
\newtheorem{example}[definition]{Example}
\newtheorem{corollary}[definition]{Corollary}

\newcommand{\bo}{{\circ}}
\newcommand{\bon}{{\bullet}}

\newcommand{\conj}{\ensuremath{\mathbin{\&}}}
\newcommand{\toto}{\leftrightarrow}

\newcommand{\tri}{\triangle}

\newcommand{\alg}[1]{{\ensuremath{\boldsymbol{ #1}}}}
\newcommand{\logic}[1]{\mathrm{ #1}}

\newcommand{\MTL}{{\ensuremath{\mathrm{MTL}}}}
\newcommand{\SMTL}{{\ensuremath{\mathrm{SMTL}}}}
\newcommand{\IMTL}{{\ensuremath{\mathrm{IMTL}}}}

\newcommand{\MTLD}{{\ensuremath{\mathrm{MTL}_\triangle}}}
\newcommand{\BL}{{\ensuremath{\mathrm{BL}}}}

\newcommand{\G}{{\ensuremath{\mathrm{G}}}}

\newcommand{\sCNG}{{\ensuremath{\mathrm{(Cong)}}}}

\newcommand{\Real}{\mathcal{R}}
\newcommand{\RC}{{\ensuremath{\Real\mathrm{C}}}}
\newcommand{\SRC}{{\ensuremath{\mathrm{S}\Real\mathrm{C}}}}
\newcommand{\FSRC}{{\ensuremath{\mathrm{FS}\Real\mathrm{C}}}}

\newcommand{\comentari}[1]{}

\newenvironment{proof}{\trivlist\item[\hskip
\labelsep{\it Proof:\/}]\ignorespaces}{\hfill$\Box$\endtrivlist}

\begin{document}

\title{Logics of formal inconsistency arising from systems of fuzzy logic} 

\author {Marcelo Coniglio$^1$ \and Francesc Esteva$^2$ \and Llu\'is Godo$^2$\\ \\
{\small $^1$ CLE and Department of Philosophy }\\
{\small State University of Campinas} \\
{\small 13083-859 Campinas, Brazil}\\
{\small \texttt{email: coniglio@cle.unicamp.br}}
\\
\\
{\small $^2$ Artificial Intelligence Research Institute (IIIA),  CSIC}\\
{\small Campus UAB, 08193 Bellaterra, Spain }\\
{\small \texttt{email: \{esteva,godo\}@iiia.csic.es}}
}

\date{}
\maketitle

\begin{abstract}
This paper proposes the meeting of fuzzy logic with paraconsistency in a very precise and foundational way. Specifically, in this paper we introduce expansions of the fuzzy logic MTL by means of primitive operators for consistency and inconsistency in the style of the so-called Logics of Formal Inconsistency (LFIs). The main novelty of the present approach is the definition of postulates for this type of operators over MTL-algebras, leading to the definition and axiomatization of a family of logics, expansions of MTL, whose degree-preserving counterpart are paraconsistent and moreover LFIs. 
\end{abstract}

\sloppy

\section{Introduction}

The well-known {\em Sorites paradox} is representative of the
problems  arising from the use of vague predicates, that is,
predicates whose extension is unclear such as `tall' and `bald'.
According to Charles S. Peirce,

\begin{quotation}
{\em A proposition is vague when there are possible states of things concerning which it is intrinsically uncertain whether, had they been contemplated by the speaker, he would have regarded them as excluded or allowed by the proposition. By intrinsically uncertain we mean not uncertain in consequence of any ignorance of the interpreter, but because the speaker's habits of language were indeterminate; so that one day he would regard the proposition as excluding, another as admitting, those states of things. } (\cite{Peirce1902})
\end{quotation}

Besides being an instigating topic for Philosophy, vagueness is also
studied from the mathematical and logical point of view. For instance, the
so-called {\em Mathematical Fuzzy Logic} (MFL), inspired by the
paradigm of {\em Fuzzy Set Theory} introduced in 1965 by L. Zadeh
(cf.~\cite{Zadeh:1965}), studies the question of vagueness from a
foundational point of view based on many-valued logics. In this sense, MFL can be considered as a degree-based approach to vagueness.\footnote{See e.g. \cite{Cintula-Fermueller-Godo-Hajek-Vagueness} for several discussions on degree-based approaches (and in particular fuzzy logic approaches) to vagueness.} Some
systems like \L ukasiewicz and G\"odel-Dummett infinitely valued
logics are, just like fuzzy sets, valued over the real interval $[0,1]$. This
supports  the idea of  MFL being as a kind of foundational
counterpart of fuzzy set theory (which is a discipline mainly devoted to
engineering applications).
The book~\cite{Hajek:1998} by P.
H\'ajek is the first monograph dedicated to a broad study of the
new subject of MFL. In that book the so-called Basic fuzzy logic \BL\ is introduced as the residuated many-valued logic with the semantics on the real unit interval  induced by all {\em continuous} t-norms and their residua. \BL\ generalizes three prominent fuzzy logics, {\L}ukasiewicz, G\"odel-Dummet and Product logics, each one capturing the semantics determined by three particular continuous t-norms, namely {\L}ukasiewicz, minimum and product t-norms respectively. 
The so-called Monoidal t-norm based logic) \MTL\  was introduced in~\cite{Esteva-Godo:Monoidal} as a generalization of \BL\ to capture the semantics induced by {\em left continuous} t-norms and their residua, in fact, as it was proved in \cite{Jenei-Montagna:StandardCompletenessMTL} the theorems of \MTL\ correspond to the common tautologies of all many-valued calculi defined by a left-continuous t-norm and its residuum. This logic, the most general residuated fuzzy logic whose semantics is based on t-norms, will be the starting point of our investigations in the present paper.

Frequently, vagueness is associated to a phenomenon of  `underdetermination of truth'. However, vagueness could be seen from an opposite perspective: if $a$ is a borderline case of a vague predicate $P$, the sentences `$a$ is $P$' and  `$a$ is not $P$'  can be both true (at least to some extent).
This leads to an interpretation of vagueness as  `overdetermination of truth', instead of underdetermination. Being so, a sentence $A$ and its negation can simultaneously be both true, without trivializing (as much we assume that not every sentence is true). This perspective, known as {\em Paraconsistent Vagueness}, connects vagueness to the subject of {\em Paraconsistent Logic} (see, for instance, \cite{HyCol2008} and~\cite{Cob2011}).

 Paraconsistency is devoted to the study  of logic systems with a negation operator, say $\neg$, such that not every contradictory set of premises $\{\varphi,\neg \varphi\}$ trivializes the system. Thus, any paraconsistent logic contains at least a contradictory but non-trivial theory. There exist several systematic approaches to paraconsistency, for instance: N. da Costa's hierarchy of C-sytems $C_n$, for $n > 0$, introduced in 1963 (see~\cite{dac:63});  Relevance  (or Relevant) logics, introduced by A. Anderson and N. Belnap in 1975 (see~\cite{Anderson-Belnap:Entailment});  the Adaptive Logics programme, developed by D. Batens and his group;  R. Routley and G. Priest's philosophical school of Dialetheism, with Priest's logic LP as its formalized counterpart (see, for instance, \cite{Priest-Routley-Norman:Paraconsistent});  and the Logics of Formal Inconsistency (LFIs), introduced by W. Carnielli and J. Marcos in  2000 (see~\cite{car:mar:02} and~\cite{car:con:mar:07}), and also studied e.g. by Avron et al. \cite{avr:zam:07,Arielietal}.  
 The main characteristic of the latter logics is that they  internalize in the object language the notions of consistency and  inconsistency by means of specific connectives (primitive or not). This constitutes a generalization of  da Costa's C-systems.

The present  paper proposes the meeting of fuzzy logic with paraconsistency in a very precise and foundational way. Specifically, we introduce extensions of the fuzzy logic \MTL\ by means of primitive operators for consistency and inconsistency, defining so LFIs based on (extensions of) \MTL. An important feature of this approach is that the LFIs defined in this manner are not based on (positive) classical logic, as in the case of most LFIs studied in the literature, including da Costa's C-systems. In particular, the LFIs proposed here do not satisfy the law of excluded middle: $\varphi \vee \neg\varphi$ is not a valid schema, in general. 

The main novelty of the present approach is the definition of postulates for primitive consistency and inconsistency fuzzy operators over the algebras associated to (extensions of) \MTL; in particular, we show how to define consistency and inconsistency operators over \MTL-algebras. This generalizes the previous approach to fuzzy LFIs introduced in~\cite{er-es-fla-go-no:2013}, where it was shown that a consistency operator can be defined in \MTLD, the expansion of \MTL\ with the Monteiro-Baaz projection connective $\triangle$. However, this consistency operator is not primitive, but it is defined in terms of the operator $\triangle$ together with other operators of \MTL. At this point, it is important to observe that \MTL,  as well as its extensions, are {\em not} paraconsistent logics, provided that the usual truth-preserving consequence relation is considered: from  $\{\varphi, \neg \varphi\}$ every other formula can be derived.
On the other hand, if a degree-preserving consequence relation is adopted, \MTL\ as well as some of its extensions become paraconsistent (see Section~\ref{prelim-1}).

The organization of this paper is as follows. In Sections~\ref{prelim-1}  and~\ref{seclfi}, the basic notions about fuzzy logics and LFIs are introduced. Then Section~\ref{secbola01} contains the main definitions and technical results. In particular, we introduce the notion of  consistency operators on \MTL-algebras and axiomatize several classes of them as expansions of \MTL.  In this framework, the question about how the consistency operator propagates with respect to the \MTL\ connectives is studied in Section~\ref{propag}. In its turn, in Section~\ref{recov} we propose a fuzzy LFI able to recover classical logic by considering additional hypothesis on the consistency operator. The dual case of inconsistency operators is briefly analyzed in Section~\ref{inconsist}. We end up with some concluding remarks  in Section~\ref{conclu}.

\section{Preliminaries I: truth-preserving and degree-preserving fuzzy logics} \label{prelim-1}

In the framework of Mathematical Fuzzy Logic there are two different families of  fuzzy logics according to how the logical consequence is defined, namely truth-preserving and degree-preserving logics. In this section we review the main definitions and properties of these two families of logics.

\paragraph{\bf Truth-preserving fuzzy logics.}

Most well known and studied  systems of mathematical fuzzy logic are the so-called {\em t-norm based fuzzy logics}, corresponding to formal many-valued calculi with truth-values in the real unit interval $[0, 1]$ and with a conjunction and an implication interpreted respectively by a (left-) continuous t-norm and its residuum respectively, and thus, including e.g. the well-known {\L}ukasiewicz  and G\"{o}del infinitely-valued logics,
corresponding to the calculi defined by {\L}ukasiewicz and $\min$ t-norms respectively. The weakest t-norm based fuzzy logic is the logic \MTL\ (monoidal t-norm based logic) introduced in \cite{Esteva-Godo:Monoidal}, whose theorems correspond to the common tautologies of all many-valued calculi defined by a left-continuous t-norm and its residuum \cite{Jenei-Montagna:StandardCompletenessMTL}.

The  language of \MTL\ consists of denumerably many propositional variables $p_1,p_2,\ldots$, binary connectives $\wedge,\conj,\to$, and the truth constant $\0$. Formulas, which will be denoted by lower case greek letters $\f,\p,\chi,\ldots$, are defined by induction as usual. Further connectives and constants are definable; in particular, $\neg\f$ stands for $\f\to\0$, $\1$ stands for $\neg\0$, $\varphi \lor \psi$ stands for $((\varphi \to \psi) \to \psi) \land ((\psi \to \varphi) \to \varphi)$, and $\varphi \leftrightarrow \psi$ stands for $(\varphi \to \psi) \land (\psi \to \varphi)$. A Hilbert-style calculus for  \MTL\  was introduced in \cite{Esteva-Godo:Monoidal} with the following set of axioms:

\begin{itemize}
\item[(A1)] $(\f \to \p) \to ((\p \to \chi) \to (\f \to \chi))$ \vspace{-0.2cm}
\item[(A2)] $\f \conj \p \to \f$  \vspace{-0.2cm}
\item[(A3)] $\f \conj \p \to \p \conj \f$  \vspace{-0.2cm}
\item[(A4)] $\f \wedge \p \to \f$  \vspace{-0.2cm}
\item[(A5)] $\f \wedge \p \to \p \wedge \f$  \vspace{-0.2cm}
\item[(A6)] $\f \conj(\f \to \p) \to \f \wedge \p$  \vspace{-0.2cm}
\item[(A7a)] $(\f \to (\p \to \chi)) \to (\f \conj \p \to \chi)$  \vspace{-0.2cm}
\item[(A7b)] $(\f \conj \p \to \chi) \to (\f \to (\p \to \chi))$  \vspace{-0.2cm}
\item[(A8)] $((\f \to \p) \to \chi) \to (((\p \to \f)\to \chi)\to \chi)$  \vspace{-0.2cm}
\item[(A9)] $\0 \to \f$
\end{itemize}
and whose unique rule of inference is {\em modus ponens}: from $\f$ and $\f \to\p$ derive $\p$.

\MTL\ is an algebraizable logic in the sense of  Blok and Pigozzi \cite{Blok-Pigozzi:AlgebraizableLogics} and its equivalent algebraic semantics is given by the class of $\logic{MTL}$-algebras, that is indeed a variety; call it $\mathbb{MTL}$. \MTL-algebras can be equivalently introduced as commutative, bounded, integral residuated lattices $\langle A, \wedge,\vee,\conj,\to,0,1\rangle$ further satisfying the following {\em prelinearity} equation: $$(x\to y)\vee(y\to x)=\1.$$

Given an \MTL-algebra  $\alg{A}=\langle A,\wedge^\alg{A},\vee^\alg{A},\conj^\alg{A},\to^\alg{A}, 0^\alg{A},1^\alg{A}\rangle$, an $\alg{A}$-evaluation is any function mapping each propositional variable into $A$, $e(\0)=0^\alg{A}$ and such that, for  formulas $\f$ and $\p$, $e(\f\wedge\p)=e(\f)\wedge^\alg{A}e(\p)$; $e(\f\vee\p)=e(\f)\vee^\alg{A}e(\p)$; $e(\f\conj\p)=e(\f)\conj^\alg{A}e(\p)$;  $e(\f\to\p)=e(\f)\to^\alg{A} e(\p)$.  An evaluation $e$ is said to be a {\em model} for a set of formulas $\Gamma$, if $e(\gamma)=1^\alg{A}$ for each $\gamma\in \Gamma$.

We shall henceforth adopt a lighter notation dropping the superscript $^\alg{A}$. The distinction between  a syntactic object and its interpretation in an algebraic structure will always be clear by the context.

The algebraizability gives the following strong completeness theorem:
\begin{itemize}
\item[] \em For every set $\Gamma\cup \{\f\}$ of formulas, $\Gamma \vdash_\MTL \f$ iff for every $\alg{A} \in \mathbb{MTL}$ and every $\alg{A}$-evaluation $e$, if $e$ is a model of $\Gamma$ then $e$ is a model of $\varphi$ as well.
\end{itemize}
For this reason, since the consequence relation amounts to preservation of the truth-constant $\1$, \MTL\ can be called a {\em (full) truth-preserving} logic.

Actually, the algebraizability is preserved for any logic L that is a (finitary)  expansion of \MTL\ satisfying the following congruence property\\

\noindent 
\begin{tabular}{lll}
\sCNG & $\f \to \p, \p\to \f \ \vdash_{\logic{L}}$ & $c(\chi_1,\dots,\chi_{i}, \f,\chi_{i+2},\dots,\chi_n)$\\
  & & $\to c(\chi_1,\dots,\chi_{i}, \p,\chi_{i+2},\dots,\chi_n)$
\end{tabular}

\

\noindent for any possible new $n$-ary connective $c$ and each $i < n$.  These expansions, that we will call {\em core expansions} of \MTL\ (in accordance with \cite{CinHorNog2014}),   are in fact {\em Rasiowa-implicative logics}  (cf.~\cite{Rasiowa:AlgebraicApproach}). As proved in \cite{Cintula-Noguera:Handbook}, every Rasiowa-implicative logic $\logic{L}$ is algebraizable and, if it is finitary, its equivalent algebraic semantics, the class $\mathbb{L}$ of $\logic{L}$-algebras, is a quasivariety. {\em Axiomatic} expansions of MTL, i.e.\ without any further inference rule, satisfying (Cong) are  called {\em core fuzzy logics} in the literature (see e.g.\ \cite{Cintula-Noguera:Handbook}), and their associated quasi-varieties of algebras are in fact varieties. 

As a consequence, any logic $\logic{L}$ which is a core expansion of MTL, in particular any core fuzzy logic, enjoys the same kind of the above  strong completeness theorem with respect to the whole class of corresponding $\logic{L}$-algebras. But for  core fuzzy logics we can say more than that. Indeed, for any core fuzzy logic $\logic{L}$,  
 the variety of $\logic{L}$-algebras can also be shown to be generated by the subclass of all its linearly ordered members \cite{Cintula-Noguera:Handbook}.\footnote{ Moreover, for a number of core fuzzy logics, including \MTL, it has been shown that their corresponding varieties are also generated by the subclass of \MTL-chains  defined on the real unit interval, indistinctively called in the literature  as {\em standard} or {\em real} chains. For instance,  \MTL\  is also complete wrt real \MTL-chains, that are of the form  $[0,1]_\ast=\langle [0,1], \min, \max, \ast,\to_\ast,0,1\rangle$ of type $\langle 2,2,2,2,0,0\rangle$, where $\ast$ denotes a left-continuous t-norm and $\to_\ast$ is its residuum \cite{Jenei-Montagna:StandardCompletenessMTL}. } This means that any core fuzzy logic $\logic{L}$ is strongly complete with respect to the class of $\logic{L}$-chains, that is, core fuzzy logics are  {\em semilinear}.

All core fuzzy logics enjoy a form of local deduction theorem. As usual, $\varphi^n$ will be used as a shorthand for $\varphi \conj \stackrel{n}{\ldots} \conj \varphi$, where $\varphi^0 = \overline{1}$.  Using this notation one can write the following {\em local deduction theorem} for any core fuzzy logics  $\logic{L}$: 
for each set of formulas $\Sigma\cup\{\varphi,\psi\}$ the following holds:
$$\Sigma,\varphi\vdash_\logic{L} \psi \mbox{ iff there is}\,\, n \geq 0 \mbox{
such that }\Sigma\vdash_\logic{L}  \varphi^n\to\psi.$$

Interesting axiomatic extensions of \MTL\ used in the paper are the ones given in Table \ref{table-logics}, but first  we list in Table \ref{axioms}  the axioms needed to define these extensions of \MTL.

\begin{center}
\begin{table}[h]

\begin{tabular}{|c|c|}
\hline Axiom schema & Name\\
\hline
\hline $\neg \neg \varphi \to \varphi$ & Involution (Inv)\\
\hline $\neg \varphi \lor ((\varphi \to \varphi \& \psi) \to \psi)$ &
Cancellation (C)\\
\hline $\varphi \to \varphi \& \varphi$ & Contraction (Con)\\
\hline $\varphi \wedge \psi \to \varphi \& (\varphi \to \psi)$ &
Divisibility (Div)\\
\hline $\varphi \land \neg \varphi \to \overline{0}$ &
Pseudo-complementation (PC)\\
\hline $(\varphi \& \psi \to \overline{0}) \lor (\varphi \land \psi
\to \varphi \& \psi)$ & Weak Nilpotent Minimum (WNM)\\
\hline
\end{tabular}
\caption{Some usual axiom schemata in fuzzy logics.}
\label{axioms}
\end{table}
\end{center}

\begin{table}[h]
\begin{center}
\begin{tabular}{|c|c|c|}
\hline Logic & Additional axiom schemata & References\\
\hline
\hline Strict \MTL\ (SMTL) & (PC) & \cite{Hajek:ObservationsMonoidal} \\
\hline Involutive \MTL\ (IMTL) & (Inv) & \cite{Esteva-Godo:Monoidal}\\
\hline Weak Nilpotent Minimum (WNM) & (WNM) & \cite{Esteva-Godo:Monoidal}\\
\hline Nilpotent Minimum (NM) & (Inv) and (WNM) &  \cite{Esteva-Godo:Monoidal}\\
\hline Basic Logic (BL) & (Div) & \cite{Hajek:1998} \\
\hline Strict Basic Logic (SBL) & (Div) and (PC) & \cite{Esteva-GHN:Involution} \\
\hline \L ukasiewicz Logic ({\L}) & (Div) and (Inv) & \cite{Hajek:1998}\\
\hline Product Logic ($\Pi$) & (Div) and (C) & \cite{Hajek-Godo-Esteva:ProductLogic} \\
\hline G\'odel Logic (\G )& (Con) & \cite{Hajek:1998}  \\
\hline
\end{tabular}
\end{center}
\caption{Some \index{axiomatic extensions of MTL}axiomatic extensions
of MTL obtained by adding the
corresponing additional axiom schemata, and the references where they
have been introduced (in the context of fuzzy
logics).}\label{table-logics}
\end{table}

\MTL\ can be considered in fact  as the  logic of left-continuous t-norms \cite{Jenei-Montagna:StandardCompletenessMTL} and \BL\ as the logic of continuous t-norms \cite{Cignoli-EGT:StandardBL},  in the sense that theorems of these logics coincide with common tautologies of interpretations on the MTL (respectively BL) chains defined on the real unit interval $[0,1]$ by left-continuous (respectively continuous) t-norms and their residua. 

Another interesting family of fuzzy logics are the so-called {\em logics of a (left-continuous) t-norm}. Given a left-continuous t-norm $\ast$, define the real (or standard) algebra $[0,1]_\ast = ([0,1], \min \max, \ast, \to_\ast, 0,1)$ where $\to_\ast$ is the residuum of $*$. Then define the logic of the t-norm $\ast$ as the logic $\logic{L}_\ast$ whose (semantical) notion of consequence relation is as follows:   $\varphi$ is a consequence of a set of formulas $\Gamma$ iff for every evaluation $v$ over $[0,1]_\ast$ such that $v(\gamma) = 1$ for each $\gamma \in \Gamma$, then $v(\varphi) = 1$. When $*$ is a continuous t-norm,  $\logic{L}_\ast$ has been proved finitely axiomatizable as extension of \BL\ (see \cite{Esteva-Godo-Montagna:EquationalCharacterizationSubvarietiesBL}).

As we have mentioned, all the axiomatic expansions of \MTL\ (i.e.\ all core fuzzy logics) are semilinear and enjoy the local deduction detachment theorem. 
Another very interesting class of fuzzy logics arise from the (non-axiomatic) expansion of  \MTL\ with the Monteiro-Baaz projection connective $\triangle$, obtaining again a finitary Rasiowa-implicative semilinear logic \MTLD.
Indeed,  \MTLD\ is axiomatized by adding to the Hilbert-style system of \MTL\ the deduction rule of necessitation (from $\varphi$ infer $\tri\varphi$) and the following axiom schemata: \\

\begin{tabular}{ll}
($\tri1$) &$\tri \varphi \lor \neg \tri \varphi$ \\
($\tri2$) & $\tri (\varphi \lor \psi) \to (\tri \varphi\lor \tri \psi)$ \\
($\tri3$) & $\tri \varphi \to \varphi$ \\
($\tri4$) & $\tri \varphi \to \tri \tri \varphi$ \\
($\tri5$) & $\tri (\varphi \to \psi) \to (\tri \varphi\to \tri \psi)$\\
\end{tabular}
\;\\

Then, one analogously defines the class of {\em $\triangle$-core fuzzy logics} as the axiomatic expansions of \MTLD\ satisfying \sCNG\ for any possible new connective. They satisfy the {\em global deduction theorem} in the following way: for any $\triangle$-core fuzzy logic $\logic{L}$, and  each set of formulas $\Sigma\cup\{\varphi,\psi\}$, the following holds:
$$\Sigma,\varphi\vdash_\logic{L} \psi \mbox{ iff }\/\Sigma\vdash_\logic{L} \tri\varphi\to \psi.$$

Semilinearity can also be inherited by many expansions of ($\triangle$-)core fuzzy logics with new (finitary) inference rules. Indeed, in \cite{Cintula-Noguera:Handbook} it is shown that an expansion $\logic{L}$ of a ($\triangle$-)core fuzzy logic is semilinear iff for each newly added finitary inference rule
\begin{description}
\item[(R)] from $\Gamma$ derive $\varphi$,
\end{description}
its corresponding $\lor$-form
\begin{description}
\item[(R$^\lor$)] from $\Gamma \lor p$ derive $\varphi \lor p$
\end{description}
is derivable  in L as well, where $p$ is an arbitrary propositional variable not appearing in $\Gamma \cup \{\f\}$.


In this paper we will use the following notions of completeness of a logic $\logic{L}$ with respect to the class {\em real} $\logic{L}$-chains. Although we will mainly focus on core fuzzy logics, we formulate them for the more general case of logics that are semilinear expansions of \MTL\  whose class of {\em real} $\logic{L}$-chains is non-empty.

\begin{definition}[\RC, \FSRC, \SRC] 
Let  {\rm L} be a semilinear core expansion of \MTL\  and let $\mathcal R$ be the class of real {\rm L}-chains, i.e. {\rm L}-chains whose support is the real unit interval $[0, 1]$. We say that {\rm L} has the \emph{(finitely) strong  $\mathcal{R}$-completeness property}, {\rm (F)}\SRC\ for short, when for every (finite) set of formulas $T$ and every formula $\varphi$ it holds that $T \vdash_{\rm L} \varphi$ iff $e(\varphi) = \overline{1}^\alg{A}$ for each $\alg{A}$-evaluation such that $e[T] \subseteq \{\overline{1}^\alg{A}\}$ for every {\rm L}-algebra $\alg{A} \in \mathcal{R}$.  We say that {\rm L} has the  $\mathcal{R}$-\emph{completeness property}, \RC\ for short, when the equivalence is true for $T = \emptyset$.
\end{definition}

Of course, the \SRC\ implies the \FSRC, and the \FSRC\ implies the \RC. The \SRC\ and \FSRC\ have traditionally been proved for many fuzzy logics by showing an embeddability property, namely by showing in the first case that every countable $\logic{L}$-chain is embeddable into a chain of $\mathcal{R}$, and in the second case by showing that every countable $\logic{L}$-chain is {\em partially} embeddable into a chain of $\mathcal{R}$ (i.e.\ for every finite partial algebra of a countable $\logic{L}$-chain there is a one-to-one mapping into some $\logic{L}$-chain over $[0, 1]$ preserving the defined operations). In~\cite{Cintula-EGGMN:DistinguishedSemantics} it was shown that, for ($\Delta$-)core fuzzy logics these sufficient conditions are also necessary (under a weak condition). This was further generalized in \cite{Cintula-Noguera:Handbook}, where Cintula and Noguera show that these conditions are also necessary for a more general class of logics, including semilinear core expansions of \MTL. 

\begin{theorem}[\cite{Cintula-EGGMN:DistinguishedSemantics,Cintula-Noguera:Handbook} Characterization of completeness properties]\label{t:Char-SKC-FSKC} \label{SKC}
Let $\logic{L}$ be a semilinear core expansion of \MTL. Then:
\begin{itemize}
\item $\logic{L}$ has the \SRC\ iff every countable $\logic{L}$-chain is embeddable into some chain of $\mathcal{R}$.
\item If the language of\/ $\logic{L}$ is finite, then $\logic{L}$ has the \FSRC\ iff every countable $\logic{L}$-chain is partially embeddable into some chain of $\mathcal{R}$.
\end{itemize}
\end{theorem}

\paragraph{\bf Degree-preserving fuzzy logics.}

It is clear that ($\triangle$-)core fuzzy logics, like \MTL, are (full) truth-preserving fuzzy logics. But besides the truth-preserving paradigm that we have so far considered, one can find an alternative approach in the literature. Given a ($\triangle$-)core fuzzy logic $\logic{L}$, and based on the definitions in~\cite{Bou-EFGGTV:PreservingDegreesResiduated}, we can introduce a variant of $\logic{L}$ that we shall denote by $\logic{L}^{\mbox{\tiny $\leq$ }}$,  whose associated  deducibility relation has the following semantics: for every set of formulas $\Gamma\cup\{\f\}$, $\Gamma\vdash_{\logic{L}^{\mbox{\tiny $\leq$ }}}\f$ iff for every $\logic{L}$-chain $\alg{A}$, every $a \in A$, and every $\alg{A}$-evaluation $v$, if $a \leq v(\p)$ for every $\p \in \Gamma$, then $a \leq v(\f)$. For this reason $\logic{L}^{\mbox{\tiny $\leq$ }}$ is known as a fuzzy logic {\em preserving degrees of truth}, or the {\em degree-preserving companion} of $\logic{L}$. In this paper, we often use generic statements about ``every logic $\logic{L}^{\mbox{\tiny $\leq$ }}$'' referring to ``the degree-preserving companion of any ($\triangle$-)core fuzzy logic (or even of any semilinear core expansion of \MTL) $\logic{L}$''.

As regards to axiomatization, if $\logic{L}$ is a core fuzzy logic, i.e.\ with Modus Ponens as the unique inference rule, then the logic $\logic{L}^{\mbox{\tiny $\leq$ }}$ admits a Hilbert-style axiomatization having  the same axioms as $\logic{L}$ and the following deduction rules \cite{Bou-EFGGTV:PreservingDegreesResiduated}:
\begin{description}
\item[(Adj-$\wedge$)] from $\f$ and $\p$ derive $\f\wedge\p$
\item[(MP-$r$)] if $\vdash_{\logic{L}}\f\to\p$ (i.e. if  $\f\to\p$ is a  theorem of $\logic{L}$), then from $\f$ derive $\p$
\end{description}
Note that if the set of theorems of $\logic{L}$ is decidable, then the above is in fact  a recursive Hilbert-style axiomatization of $\logic{L}^{\mbox{\tiny $\leq$ }}$.

In general, let $\logic{L}$ be a semilinear core expansion of \MTL\  with a set of new inference rules,
\begin{description}
\item[(R$_i$)] from $\Gamma_i$ derive $\varphi_i$, for each  $i\in I$.
\end{description}
Then $\logic{L}^{\mbox{\tiny $\leq$ }}$ is axiomatized by adding to the axioms of $\logic{L}$ the above two inference rules plus the following restricted rules
\begin{description}
\item[(R$_i$-$r$)] if  $\vdash_{\logic{L}} \Gamma_i$, then derive $\varphi_i$.
\end{description}

Moreover, if $\logic{L}$ is a $\triangle$-core fuzzy logic, then the only rule one should add to $\logic{L}^{\mbox{\tiny $\leq$ }}$ is the following restricted necessitation rule for $\triangle$:
\begin{description}
\item[($\triangle$-$r$)] if $\vdash_{\logic{L}} \f$, then derive $\triangle \f$.
\end{description}

The  key relationship between $\logic{L}$ and $\logic{L}^{\mbox{\tiny $\leq$ }}$ is given by the following equivalence:
for any formulas $\varphi_1, \ldots, \varphi_n, \psi$, it holds
\begin{center}
$\varphi_1, \ldots, \varphi_n \vdash_{\logic{L}^{\mbox{\tiny $\leq$ }}} \psi$ iff \/ $\vdash_\logic{L} (\varphi_1 \land \ldots \land \varphi_n) \to \psi$.
\end{center}
This relation points out that, indeed, deductions from a finite set of premises in $\logic{L}^{\mbox{\tiny $\leq$ }}$ exactly correspond to  theorems in $\logic{L}$. In particular, both logics share the same theorems: $\vdash_{\logic{L}} \varphi$ \mbox{}  iff \mbox{}  $\vdash_{\logic{L}^{\mbox{\tiny $\leq$ }}} \varphi$. Moreover, this also implies that if $\logic{L}'$ is a conservative expansion of $\logic{L}$, then $\logic{L'}^{\mbox{\tiny $\leq$ }}$ is also a conservative expansion of $\logic{L}^{\mbox{\tiny $\leq$ }}$.

\section{Preliminaries II: logics of formal inconsistency}   \label{seclfi}

Paraconsistency is the study of logics having a negation operator $\neg$ such that it is not explosive with respect to $\neg$, that is, there exists at least a formula $\varphi$ such that from $\{\varphi, \neg\varphi\}$ it does not follow any formula. In other words, a paraconsistent logic is a logic having at least a contradictory, non-trivial theory.

Among the plethora of paraconsistent logics proposed in the literature, the {\em Logics of Formal
Inconsistency} (LFIs), proposed in~\cite{car:mar:02} (see also~\cite{car:con:mar:07}), play an important role, since they internalize in the object language the very notions of consistency and  inconsistency by
means of specific connectives (primitives or not).\footnote{We should warn the reader that in the frame of LFIs, the term {\em consistency} is used to refer to formulas that basically exhibit a classical, explosive behaviour  rather than for referring to formulas being (classically) satisfiable. } This
generalizes the strategy of N. da Costa, which introduced in~\cite{dac:63} the
well-known hierarchy of systems $C_n$, for $n > 0$.
Besides being able to distinguish between
contradiction and inconsistency, on the one hand, and non-contradiction and consistency, on the other, LFIs are non-explosive logics, that is, paraconsistent:  in general, a contradiction does not entail arbitrary
statements, and so the  Principle of Explosion (for all $\varphi, \psi$ it holds $\varphi, \neg\varphi\vdash\psi$)  does not hold. However, LFIs are {\em gently explosive}, in the sense that, adjoining
the additional requirement of consistency, then contradictoriness
does cause explosion: $\bigcirc(\varphi),\varphi, \neg\varphi\vdash\psi$ for every $\varphi$ and $\psi$. Here, $\bigcirc(\varphi)$ denotes that $\varphi$ is consistent. The general definition of LFIs we will adopt here, slightly modified from the original one proposed in~\cite{car:mar:02} and~\cite{car:con:mar:07}, is the following:

\begin{definition} \label{defLFI}
Let  $\logic{L}$ be a logic defined in a language $\mathcal{L}$ containing a negation $\neg$, and let  $\bigcirc(p)$ be a nonempty set of formulas depending exactly on the propositional variable $p$. Then $\logic{L}$ is an LFI (with respect to $\neg$ and  $\bigcirc(p)$) if the following holds (here, $\bigcirc(\varphi)=\{\psi[p/\varphi]  \ : \  \psi(p) \in \bigcirc(p)\}$ and $\psi[p/\varphi]$ denotes the formula obtained from $\psi$ by replacing every occurrence of the variable $p$ by the formula $\varphi$):
\begin{itemize}
       \item [(i)] $\varphi,\neg\varphi\nvdash\psi$ for some $\varphi$ and $\psi$, i.e., $\logic{L}$ is not explosive w.r.t.  $\neg$;
       \item [(ii)] $\bigcirc(\varphi),\varphi \nvdash \psi$ for some $\varphi$ and $\psi$;
       \item [(iii)] $\bigcirc(\varphi), \neg \varphi \nvdash \psi$ for some $\varphi$ and $\psi$; and
       \item [(iv)]  $\bigcirc(\varphi),\varphi,\neg\varphi\vdash\psi$ for every $\varphi$ and $\psi$, i.e., $\logic{L}$ is gently explosive w.r.t.  $\neg$ and $\bigcirc(p)$.
\end{itemize}
\end{definition}

In the case that $\bigcirc(\varphi)$ is a singleton (which will be the usual situation), its element will we denoted by $\bo\varphi$, and $\bo$ will be called a {\em consistency} operator in $\logic{L}$ with respect to $\neg$. A consistency operator can be primitive (as in the case of most of the systems treated in~\cite{car:mar:02} and~\cite{car:con:mar:07}) or, on the contrary, it can be defined in terms of the other connectives of the language. For instance, in the well-known system $C_1$ by da Costa, consistency is defined by the formula $\bo\varphi=\neg(\varphi \wedge \neg\varphi)$ (see~\cite{dac:63}).

Given a consistency operator $\bo$, an {\em inconsistency} operator $\bullet$ is naturally defined as $\bullet\varphi=\neg\bo\varphi$. In the stronger LFIs, the other way round holds, and so $\bo$ can be defined from a given $\bullet$ as $\bo\varphi=\neg\bullet\varphi$.

All the LFIs proposed in~\cite{car:mar:02} and~\cite{car:con:mar:07} are extensions of positive classical logic, therein called $\mathrm{CPL}^+$. The weaker system considered there is called  $\mathrm{mbC}$, defined in a language containing $\wedge$, $\vee$, $\to$, $\neg$ and $\bo$, and it is obtained from  $\mathrm{CPL}^+$ by adding the schema axioms $\varphi \vee \neg\varphi$ and $\bo\varphi \to(\varphi \to(\neg\varphi \to\psi))$.

As we shall see in the next section, the definition of LFIs can be generalized to the algebraic framework of MTLs, constituting an interesting approach to paraconsistency under the perspective of LFIs, but without the requirement of being an extension of  $\mathrm{CPL}^+$.

\section{Axiomatizating expansions of paraconsistent fuzzy logics with consistency operators $\bo$} \label{secbola01}

As observed in \cite{er-es-fla-go-no:2013}, truth preserving fuzzy logics are not paraconsistent since from $\varphi, \neg \varphi$ we obtain $\varphi \& \neg \varphi$, that is equivalent to the truth-constant $\overline{0}$, and thus they are explosive. However in the case of degree-preserving fuzzy logics,   from $\varphi, \neg \varphi$  one cannot   always  derive  the truth-constant $\overline{0}$, and  hence  there are  paraconsistent degree-preserving fuzzy logics. Indeed we have the following scenario.

\begin{proposition} Let $\logic{L}$ be a semilinear core expansion of \MTL.
The following conditions hold:
\begin{enumerate}
\item  $\logic{L}$ is  explosive, and hence it is not paraconsistent

\item  $\logic{L}^{\mbox{\tiny $\leq$}}$ is paraconsistent iff $\logic{L}$ is not pseudo-complemented, i.e. if $\logic{L}$  does not prove the law $\neg(\varphi \land \neg \varphi)$.
\end{enumerate}
\end{proposition}

The proof of the second item is easy since  $\varphi, \neg \varphi  \vdash_{\logic{L}^{\mbox{\tiny $\leq$}}} \overline{0}$ does not hold only in the case $\logic{L}$ does not prove $(\varphi \land \neg \varphi) \to \overline{0}$, or in other words, only in the case $\logic{L}$ is not an expansion of \SMTL.

As a consequence, from now on $\logic{L}$ will refer to any semilinear core expansion of \MTL\ which is not a \SMTL\ logic (not satisfying axiom (PC)). Indeed we are interested in the expansion with a consistency operator $\bo$ of a logic $\logic{L}^{\mbox{\tiny $\leq$ }}$ (when $\logic{L}$ is not a \SMTL\ logic).  In order to axiomatize these expansions, we need first to axiomatize the expansion of the truth-preserving $\logic{L}$ with such an operator $\bo$ and from them, as explained in Section \ref{prelim-1}, we can then obtain the desired axiomatizations.

\subsection{Expansions of truth-preserving fuzzy logics with consistency operators $\bo$}

Having in mind the properties that a consistency operator has to verify in a paraconsistent logic (recall Definition~\ref{defLFI}), and taking into account that any semilinear core expansion of \MTL\ is complete with respect to the chains of the corresponding varieties, it seems reasonable to define a consistency operator over a non-\SMTL\ chain $\bf A$  as a unary operator $\bo:A \to A$ satisfying the following conditions:

\begin{itemize}
\item[(i)] $x \wedge \bo(x) \neq 0$ for some $x \in A$;
\item[(ii)] $\neg x\wedge \bo(x) \neq 0$ for some $x \in A$;
\item[(iii)] $x \wedge \neg x \wedge \bo(x) = 0$ for every $x \in A$.
\end{itemize}

Such an operator $\bo$ can be indeed considered as the algebraic counterpart of a {\em consistency operator}  in the sense of Definition~\ref{defLFI}. Actually, we can think about the value $\bo(x)$ as denoting the (fuzzy) degree of `classicality' (or `reliability', or `robustness') of $x$ with respect to the satisfaction of the law of explosion, namely $x \wedge \neg x=0$. \\

Let us have a closer look at how operators $\bo$ on a non-\SMTL\ chain satisfying the above conditions (i), (ii) and (iii) may look like.  Let us consider the set $N({\bf A}) = \{x \in A\setminus\{1\} \mid \neg x = 0\}$. Notice that either $N({\bf A}) = \emptyset$ (for example, this is the case of IMTL chains) or $N({\bf A}) \in \{ [a, 1), (a,1)\}$ where $a = \bigwedge N({\bf A})$. If $x \notin N(A) \cup \{0,1\}$ then by (iii) we have $\bo(x) = 0$. Thus since $\neg(x) = 0$ for $x \in N(A)$, (ii) implies that $\bo(0) > 0$. On the other hand by (i), $\bo(x) >0$ for some $x \in N(A) \cup \{�1\}$.
Therefore, any operator $\bo$ verifying (i), (ii) and (iii) must satisfy the following minimal conditions:\\
$$  \left \{ \begin{array}{ll}
\bo(0) > 0, & \\
\bo(x)  = 0, & \mbox{if } x \in (0,1) \setminus N({\bf A})\\
\bo(x) > 0, &  \mbox{for some } x \in N(A) \cup \{�1\} \\
\end{array}
\right .
$$

However,  since in our setting the intended meaning of $\bo(x)$ is the (fuzzy) degree of `classicality' or `reliability', or `robustness' of $x$, we propose the following stronger postulates for such a consistency operator on non-\SMTL\ chains $\bf A$:

\begin{itemize}
\item[(c1)] If $x \wedge \neg x \neq 0$ then $\bo(x)=0$;
\item[(c2)] If $x \in\{0,1\}$ then $\bo(x)=1$;
\item[(c3)] If $\neg x=0$ and $x \leq y$  then ${\circ}(x) \leq {\circ}(y)$.
\end{itemize}

Clause (c1) just guarantees that condition (iii) for consistency operators  is satisfied by $\bo$.
In the classical case, both truth-values $0$ and $1$ satisfy the explosion law $x \wedge \neg x = 0$ and so $\bo(x)=1$ for every truth-value $x$. Since $\bo$ intends to extend the classical case, clause (c2) reflects this situation (another justification for (c2) is that $0$ and $1$ are classical truth-values with fuzzy degree $1$). Moreover, clause (c2) ensures that conditions (i) and (ii) for consistency operators are satisfied.  Finally, clause (c3) ensures the coherency of $\bo$: in $N({\bf A})$, the segment of the chain where $\bo$ is positive, the consistency operator $\bo$ is monotonic, in accordance with the idea that $\bo(x)$ is the fuzzy degree of classicality, from the perspective of the explosion law: ``the closer is $x$  to $1$, the more classical is $x$".  In Figure \ref{neg-bola}, we depict in blue (dashed-line) the graph of the negation $\neg$ in the real BL-chain $[0, 1]_*$, where * is the ordinal sum of {\L}ukasiewicz t-norm in $[0, a]$ and another arbitrary t-norm on the interval $[a, 1]$ and in red (bold line) the graph  of  a  $\bo$ operator compatible with the above postulates.

\begin{figure}[htbp]
\begin{center}
\includegraphics[width=6cm]{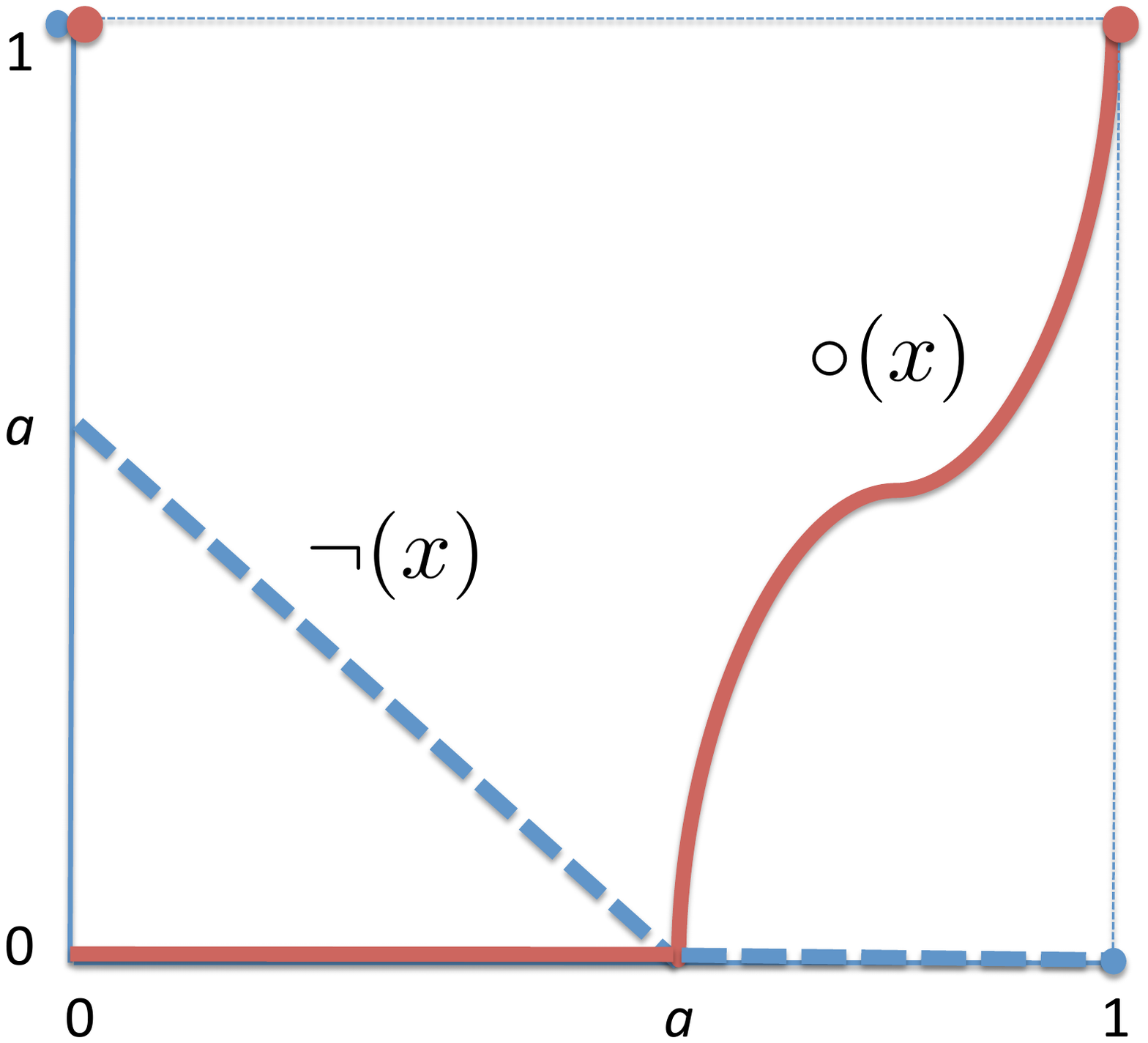}
\caption{Graph of the negation in a BL-algebra of the form $[0, a]_{\textrm{\L}} \oplus [a, 1]_*$ and a graph of a $\bo$ operator satisfying postulates (c1), (c2) and (c3).}
\label{neg-bola}
\end{center}
\end{figure}

As a consequence, we propose the following definition.

\begin{definition} Let  $\logic{L}$ be any semilinear core expansion of \MTL.  Given an axiomatization of $\logic{L}$, we define the logic $\logic{L}_\bo$ as the expansion of $\logic{L}$ in a language which incorporates a new unary connective $\bo$ with  the following axioms: \vspace{0.2cm}

\begin{tabular}{ll}
{\rm(A1)} & $\neg(\varphi\wedge \neg\varphi \wedge \bo\varphi)$\\
{\rm(A2)}  & ${\circ} \bar 1$\\
{\rm(A3)} & ${\circ} \bar 0$\\
\end{tabular}
\vspace{0.2cm}

\noindent
and the following inference rules:  \vspace{0.2cm}

\begin{tabular}{ll}
{\sCNG} \, $\displaystyle \frac{ (\varphi \leftrightarrow \psi) \vee \delta}{ (\bo\varphi \leftrightarrow \bo\psi) \vee \delta}$ \hspace{1cm} & {\rm(Coh)} \,$\displaystyle \frac{(\neg\neg\varphi \wedge (\varphi\to \psi)) \vee \delta}{ (\bo\varphi \to \bo\psi) \vee \delta}$  \vspace{0.4cm} \\
\end{tabular}

\end{definition}

Due to the presence of the rule \sCNG,  $\logic{L}_\bo$ is a Rasiowa-implicative logic, and thus it  is algebraizable in the sense of Blok and Pigozzi, and its algebraic semantics is given by $\logic{L}_\bo$-algebras.

\begin{definition}  $\logic{L}_\bo$-algebras are expansions of $\logic{L}$-algebras with a new unary operation $\bo$ satisfying the following conditions, for all $x, y, z\in A$:
\begin{enumerate}
\item[($\bo 1$)] $ x \wedge \neg x  \wedge \bo(x) = 0$ 
\item[($\bo 2$)] $\bo(1) = \bo(0) = 1$    
\item[($\bo 3$)] if  $(\neg\neg x \wedge (x \to y)) \vee z = 1$ then $(\bo(x) \to \bo(y)) \vee z = 1$     
\end{enumerate}
\end{definition}

Thus, the class $\mathcal{L}_\bo$ of $\logic{L}_\bo$-algebras is a quasivariety,  and since it is the equivalent algebraic semantics of the logic  $\logic{L}_\bo$,  $\logic{L}_\bo$ is (strongly) complete with respect to $\mathcal{L}_\bo$. But since the inference rules \sCNG\ and (Coh) are closed under $\lor$-forms, we know (see Section \ref{prelim-1})
that $\logic{L}_{\bo}$ is also semilinear and hence it is complete with respect to the class of   $\logic{L}_\bo$-chains.

\begin{proposition}[Chain completeness]
The logic  $\logic{L}_\bo$ is strongly complete with respect to the class of $\logic{L}_\bo$-chains.
\end{proposition}

It is worth pointing out that the above conditions on $\bo$ in a linearly ordered $\logic{L}_\bo$-algebra faithfully capture the three intended properties (c1)-(c3) that were required to such $\bo$ operator at the beginning  of this section.
Indeed, one can easily show the following lemma.

\begin{lemma} Let $\bf A$ be a L-chain and let $\bo: A \to A$ a mapping. Then $\bo$ satisfies conditions (c1), (c2) and (c3) iff $\bf A$ expanded with $\bo$ is a $\logic{L}_\bo$-chain.
\end{lemma}

\begin{proof} The implication from left to right is immediate since each condition (c$i$) implies condition $(\bo i)$ for $i = 1,2,3$, actually (c2) $= (\bo 2)$. For the other direction, it is enough to observe that in a chain it holds that $x \land y  = 0$ iff $x = 0$ or $y = 0$, and $x \lor y = 1$ iff either $x = 1$ or $y = 1$.  Then it is obvious that $(\bo1)$ and $(\bo 3)$ are indeed equivalent to (c1) and (c3) respectively.
\end{proof}

\begin{example} (1) Let $\logic{L}$ be  the logic  of a t-norm  which is an ordinal sum of a {\L}ukasiewicz component and a  G\"odel component with an  idempotent separating point $0 < a < 1$ (a non-$SMTL$ chain denoted $\L \oplus G$ such that $N(\L \oplus G) = [a,1)$). Then an $\bo$ operator in the corresponding standard algebra is any function $\bo: [0, 1] \to [0, 1]$  such that :
\begin{itemize}
\item[(i)] $\bo(x) = 1$  if $x \in \{0,1\}$, 
\item[(ii)] $\bo(x) = 0$ if $x \in (0,a)$ (where $x \land \neg x \neq 0$),
\item[(iii)] $\bo$ is  not decreasing in $N(\L \oplus G) = [a,1)$ (where $\neg x = 0$).
\end{itemize}
Therefore there are as many consistency operators as non-decreasing functions over the interval $[a,1]$ with values in $[0, 1]$.

(2) Let $\logic{L} = \L$ be   {\L}ukasiewicz logic, i.e. the logic of the {\L}ukasiewicz t-norm complete with respect to the standard chain $[0,1]_{\textrm{\L}}$. Since the negation is involutive, we have $N([0,1]_{\textrm \L}) = \emptyset$, and thus  there is a unique $\bo$ operator definable on the \L ukasiewicz standard chain: the one defined as $\bo(x) = 1$ if $x \in \{0,1\}$, and $\bo(x) = 0$ otherwise.
\end{example}

We can now prove that the logic  $\logic{L}_\bo$ is a conservative expansion of  $\logic{L}$ in the following strong sense.

\begin{proposition}[Conservative expansion]\label{p:ConsExp}
Let $\mathcal{L}$ be the language of\/ $\logic{L}$. For every set $\Gamma \cup \{\varphi\}$ of $\mathcal{L}$-formulas, $\Gamma \vdash_{\logic{L}_{\bo}} \varphi$ iff $\Gamma \vdash_{\logic{L}} \varphi$.
\end{proposition} 

\begin{proof}
One implication is trivial. For the other one, assume that $\Gamma \nvdash_{\logic{L}} \varphi$. Then there exists an $\logic{L}$-chain $\alg{A}$ and an $\alg{A}$-evaluation $e$ such that $e[\Gamma] \subseteq \{\overline{1}\}$ and $e(\varphi) \neq \overline{1}$. $\alg{A}$ can be expanded to an $\logic{L}_{\bo}$-chain $\alg{A}'$ e.g.\ by defining $\bo(1)=\bo(0) = 1$ and $\bo(x) = 0$ for every $x \in A \setminus \{0,1\}$. Then $\alg{A}'$ and $e$ provide a counterexample in the expanded language showing that $\Gamma \nvdash_{\logic{L}_{\bo}} \varphi$.
\end{proof}

\begin{theorem}[Strong real completeness]\label{SRC}
The logic  $\logic{L}_{\bo}$ has the \SRC\ if, and only if, $\logic{L}$ has the \SRC.
\end{theorem}

\begin{proof}
Again, one implication just follows from the fact that $\logic{L}_{\bo}$ is a conservative expansion of $\logic{L}$. For the converse one assume that $\logic{L}$ has the \SRC. We have to show that any countable $\logic{L}_{\bo}$-chain can be embedded into a standard $\logic{L}_{\bo}$-chain. Let  $\alg{A}$ be  a countable $\logic{L}_{\bo}$-chain. By Theorem~\ref{t:Char-SKC-FSKC}, we know that the $\bo$-free reduct of $\alg{A}$ is embeddable into a (standard) $\logic{L}$-chain $\alg{B}$ on $[0, 1]$. Denote this embedding by $f$ and  define $\bo'\colon [0, 1] \to [0, 1]$ in the following way:

\begin{itemize}
\item[(i)] $\bo'(0) = \bo'(1) = 1$
\item[(ii)] $\bo'(x) = 0$ for $x$ such that $x > 0$ and $\neg x > 0$.
\item[(iii)] $\bo'$ restricted to the interval $\{ x \in (0, 1) \mid  \neg x = 0\}$ is defined as $\bo'(z) = \sup \{f(\bo(x))  \ :  \ x \in A, f(x) \leq z\}$
\end{itemize}
So defined, $\bo'$ is non-decreasing on $\{ x \in [0, 1]  \ : \  \neg x = 0\}$ such that ${\bo}'(f(x)) = f(\bo(x))$ for any $x \in A$ and hence $\alg{B}$ expanded with $\bo'$ is a standard $\logic{L}_{\bo}$-chain where $\alg{A}$ is embedded.
\end{proof}

Taking into account that for a (semilinear core expansion of \MTL) logic $\logic{L}$ being finite strong real completeness is equivalent to the fact that every countable $\logic{L}$-chain is partially embeddable into some  $\logic{L}$-chain over $[0, 1]$ (see Theorem \ref{SKC}), the following corollary can be easily proved by the same technique used in the above Theorem \ref{SRC}.

\begin{corollary}[Finite strong standard completeness]\label{FSRC}
A logic $\logic{L}$ has the \FSRC\ if, and only if, $\logic{L}_{\bo}$ has the \FSRC.
\end{corollary}

\subsection{Some interesting extensions of the logics $\logic{L}_\bo$} \label{extensions}

As shown in the examples above of operators $\bo$ in $\logic{L}_\bo$-chains $\bf A$, these operators are completely determined over the set $\{ x \in A \ : \ x \land \neg x \neq 0\}$, but they can be defined in different ways in the interval where $\neg x = 0$. In this section we first consider adding a consistency operator  $\bo$ to logics whose associated chains have no elements $x < 1$ such that $\neg x = 0$ (chains {\bf A} such that $N({\bf A}) = \emptyset$). These logics can be obtained from any $\logic{L}$ by adding a suitable inference rule, and will be denoted as $\logic{L}^{\neg \neg}$.
In the second subsection we focus on logics $\logic{L}_\bo$ where $\bo$ is crisp, and in particular we consider the two extremal cases of these operators, namely those such that $\bo(x) = 0$ for all  $x \in N({\bf A})$ and those such that  $\bo(x) = 1$ for all $x \in N({\bf A})$.

\subsubsection{The case of $\logic{L}^{\neg \neg}$ logics with $\bo$ operators}

In this subsection we study the case of logics $\logic{L}_\bo$ whose associated $\logic{L}$-chains are those where $\neg x = 0$ necessarily implies  that  $x = 1$. First, from a logic $\logic{L}$ we will define the logic $\logic{L}^{\neg \neg}$ and then we will add the consistency operator.

The logic $\logic{L}^{\neg \neg}$ is defined as the extension of L by adding the following rule:
\begin{itemize}
\item[] {\rm($\neg\neg$)} \, $\displaystyle \frac{ \neg \neg \varphi}{ \varphi}$
\end{itemize}

Obviously $\logic{L}^{\neg \neg}$ is complete with respect to the corresponding quasi-variety of $\logic{L}^{\neg \neg}$-algebras, that is, the class of $\logic{L}$-algebras satisfying the quasi-equation ``If $\neg\neg x = 1$ then $x = 1$'', or equivalently the quasi-equation ``If $\neg x = 0$ then $x = 1$''.

\begin{remark}
In general,  the class of $\logic{L}^{\neg \neg}$-algebras is not a variety. For instance, in \cite{Noguera-Esteva-Gispert:SomeVarietiesMTL} it is proved that the class of  weak nilpotent minimum algebras satisfying the quasi-equation $\neg  x = 0 \Rightarrow x = 1$ is a quasi-variety that is not a variety.
For instance, take the WNM-chain $\bf C$ over the real unit interval defined by the negation:
$$n(x) = \left \{
\begin{array}{ll}
\vspace{0.5 cm}
1-x & {\rm if} \ x \in [0,\frac{1}{5}] \cup [ \frac{4}{5},1]\\
\vspace{0.5 cm}
\frac{1}{5}, & {\rm if} \ x \in [\frac{3}{5},\frac{4}{5}]\\
\frac{4}{5} -x, & {\rm if} \ x \in [\frac{1}{5},\frac{3}{5}]
\end{array}
\right .$$
Take the filter $F = [ \frac{4}{5},1]$. Then an easy computation shows that the quotient algebra ${\bf C}/\equiv_F$ is isomorphic to the standard WNM-chain ${\bf C}_F$ defined by the negation:

$$n_F(x) = \left \{
\begin{array}{ll}
\vspace{0.5 cm}
0, & {\rm if} \ x \in [\frac{2}{3},1]\\
\frac{2}{3}-x, & {\rm if} \ x \in (0,\frac{2}{3}]
\end{array}
\right .$$
\vspace{0.2cm}

\noindent Clearly  $\{x \in C \ : \ \neg x = 0\} = \{1\}$ but $\{x \in C_F  \ : \ \neg x = 0\} \neq \{1\}$, i.e.\ $\bf C$ belongs to the quasi-variety of WNM$^{\neg \neg}$-algebras but ${\bf C}_F$ does not, so the class of $\logic{L}^{\neg \neg}$-algebras is not closed by homomorphisms.

\end{remark}

Moreover $\logic{L}^{\neg \neg}$ is a semilinear logic since it satisfies the following proposition.

 \begin{lemma}\label{p:neg}
The following rule
\begin{itemize}
\item[] {\rm($\neg\neg$)$^\lor$} \, $\displaystyle \frac{ \neg \neg \varphi \lor \delta}{ \varphi \lor \delta}$
\end{itemize}
is derivable in $\logic{L}^{\neg \neg}$.
\end{lemma}

\begin{proof}
Since $\delta \to \neg\neg \delta$ is a theorem of MTL, it is clear that $ \neg \neg \varphi \lor \delta \vdash_{\logic{L}} \neg \neg \varphi \lor \neg \neg \delta$, and so $ \neg \neg \varphi \lor \delta \vdash_{\logic{L}} \neg \neg (\varphi \lor \delta)$ as well. Then, by using the rule $(\neg \neg)$ we have that  $\varphi \lor \delta$ is derivable in $\logic{L}^{\neg \neg}$ from the premise $\neg \neg \varphi \lor \delta$.
\end{proof}

\begin{corollary}[Chain completeness]
The logic  $\logic{L}^{\neg \neg}$ is semilinear and thus strongly complete with respect to the class of $\logic{L}^{\neg \neg}$-chains.
\end{corollary}

\begin{proof}
Since the inference rule $(\neg \neg)$ is  closed under $\lor$-forms, we know
that $\logic{L}^{\neg \neg}$ is also semilinear (see Section \ref{prelim-1}) and hence it is complete with respect to the class of   $\logic{L}^{\neg \neg}$-chains
\end{proof}
\begin{remark} Obviously, if $\logic{L}$ is an IMTL logic (i.e. a logic where its negation is involutive), then $\logic{L}^{\neg\neg} = \logic{L}$. Also, for interested readers, we could notice that $\logic{BL}^{\neg \neg}$ is actually {\L}ukasiewicz logic {\L} since the only BL-chains satisfying the quasi-equation ``if $\neg x = 0$ then $x = 1$'' are the involutive BL-chains, i.e MV-chains. This is not the case for $\logic{MTL}^{\neg \neg}$ which is not equivalent to IMTL {\rm(see the WNM logic defined in the previous remark, that satisfy rule {\rm($\neg\neg$)} and it is not \IMTL.)}.
\end{remark}
Now we add the consistency operator $\bo$ to the logic $\logic{L}^{\neg \neg}$. By this we mean to expand the language with  an unary connective $\bo$ and to add  the axioms  (A1), (A2) and (A3) and the inference rules \sCNG\ and (Coh). Obviously, the resulting logic $\logic{L}_\bo ^{\neg \neg}$ is complete with respect to the quasi-variety of $\logic{L}_\bo^{\neg \neg}$-algebras and with respect to the class of chains of the quasi-variety. The completeness theorems with respect to real chains
also apply to $\logic{L}_\bo ^{\neg \neg}$. Moreover we can easily prove that the following schemes and inference rule are provable and derivable respectively in $\logic{L}_\bo^{\neg \neg}$:
\begin{itemize}
\item[(B1)] $\neg \bo\varphi \lor \varphi \lor \neg \varphi$,
\item[(B2)] $\bo(\varphi \toto \psi) \to (\bo \varphi \toto \bo \psi)$,
\item[(B3)] $\bo(\varphi \lor \psi) \to \bo \varphi \lor \psi$,
\item[(B4)] $\bo \bar{0}$
\item[] \hspace{-1cm}{($\bo$Nec)} \, $\displaystyle \frac{  \varphi}{\bo\varphi}$
\end{itemize}

These properties allows us to provide a simpler axiomatization of $\logic{L}_\bo ^{\neg \neg}$.

\begin{theorem}
$\logic{L}_\bo ^{\neg \neg}$ can be axiomatized by adding to the axiomatization of $\logic{L}^{\neg \neg}$ the axioms (B1)-(B4) and the rule ($\bo$Nec).
\end{theorem}

\begin{proof} Let us denote by $\logic{L}^+_\bo$ the resulting new system in the expanded language with $\bo$ obtained from $\logic{L}^{\neg \neg}$ by adding the axioms (B1)-(B4) and the rule ($\bo$Nec).
The axioms (B1)-(B4) and the rule ($\bo$Nec) are clearly sound wrt  $\logic{L}_\bo ^{\neg \neg}$-algebras. Thus we need only to prove that axioms of $\logic{L}_\bo ^{\neg \neg}$ are provable in the new system $\logic{L}^+_\bo$, and that the rules \sCNG\ and (Coh) are also admissible in $\logic{L}^+_\bo$. It is obvious that from (B1) we can obtain (A1), since  $\neg \bo\varphi \lor \varphi \lor \neg \varphi$ implies $\neg \bo\varphi \lor \neg \neg \varphi \lor \neg \varphi$ and the latter is equivalent to (A1). (A2) is an easy consequence of rule 
($\bo$Nec), and (A3) is (B4).
Thus it only remains to prove that  \sCNG\ and (Coh) are derivable  in $\logic{L}^+_\bo$ (in what follows $\vdash$ stands for $\vdash_{\logic{L}^+_\bo}$).

On the one hand, from  $ (\varphi \leftrightarrow \psi) \vee \delta$, using rule ($\bo$Nec), we obtain $ \bo ((\varphi \leftrightarrow \psi) \vee \delta)$, and by (B3) and MP, $ \bo (\varphi \leftrightarrow \psi) \vee \delta$. Finally, by (B2), MP and taking into account the monotonicity of $\lor$, we get $ (\bo \varphi \toto \bo \psi) \vee \delta$. Hence \sCNG\ is derivable.

On the other hand,  from $ \neg \neg \varphi \lor \delta$ and $ (\varphi \to \psi) \lor \delta$, using  $(\neg \neg)^{\lor}$, we obtain $ \varphi \lor \delta$ and $ (\varphi \to \psi) \lor \delta$,  and thus $ ( \varphi \lor \delta) \& ((\varphi \to \psi) \lor \delta)$ as well. Therefore by properties of $\&$, we get $ (\varphi \& ((\varphi \to \psi) \lor \delta)) \lor (\delta \& ((\varphi \to \psi) \lor \delta))$, and by MP and monotonicity for $\&$, we obtain $\psi \lor \delta$. Now,  by ($\bo$Nec) it follows $ \bo(\psi \lor \delta)$ and by (B3),  $ \bo\psi \lor \delta$. 
Since \MTL\ proves the schema $\alpha \to (\beta \to \alpha)$ then 
$\bo\psi  \to (\bo\varphi \to \bo\psi)$ is a theorem of $\logic{L}_\bo^+$. Thus, by monotonicity of $\lor$ and modus ponens, we obtain $ (\bo \varphi \to \bo \psi) \lor \delta$. 
Therefore (Coh) is a derivable rule in MTL$_\bo ^{\neg \neg}$, and hence  in L$_\bo ^{\neg \neg}$ as desired.
\end{proof}

Taking into account that $\logic{L}_\bo ^{\neg \neg}$ is chain-complete, it is  interesting to check how operators $\bo$ can be defined in a $\logic{L}^{\neg \neg}$-chain {\bf A}. Indeed, since in this case $N({\bf A}) = \emptyset$, the $\bo$ operator is completely determined and defined as:
$$\bo(x) = \left \{
\begin{array}{ll}
1, & {\rm if} \ x \in \{0,1\} \\
0, & {\rm otherwise}
\end{array}
\right .$$

The interested reader will have observed that such an operator can also be defined in the algebras of the logic $\logic{L}_\triangle$ (the expansion of $\logic{L}$ with the Monteiro-Baaz $\triangle$ operator) as $\bo(x) = \triangle (x \lor \neg x)$ (cf.~\cite{er-es-fla-go-no:2013}). And conversely, in $\logic{L}_\bo ^{\neg \neg}$-algebras the $\triangle$ operator is also definable as $\triangle x = \bo(x) \land x$. Therefore the following result is easy to prove using chain completeness results for both logics.

\begin{corollary}
$\logic{L}_\bo^{\neg \neg}$-algebras and $(\logic{L}_\triangle)^{\neg \neg}$-algebras are termwise equivalent, hence the logics  $\logic{L}_\bo^{\neg \neg}$ and $(\logic{L}_\triangle)^{\neg \neg}$ themselves are  equivalent.
\end{corollary}

As a consequence, let us mention that, unlike $\logic{L}^{\neg \neg}$, the class of $\logic{L}_\bo^{\neg \neg}$-algebras is always a variety, since this is clearly the case of $(\logic{L}_\triangle)^{\neg \neg}$: indeed, the rule $(\neg\neg)$ can be equivalently expressed in $(\logic{L}_\triangle)^{\neg \neg}$ as the axiom $\triangle(\neg\neg \varphi) \to \varphi$. 

\subsubsection{Logics with crisp consistency operators: minimal and maximal consistency operators}

As previously observed, the consistency operator $\bo$ is non-decreasing in the segments of the chains where $\neg x=0$, producing a kind of `fuzzy degree of classicality'. In the previous section we have analyzed a special case where the operator $\bo$ is crisp in the sense that it takes only the values $0$ and $1$. The aim of this section is to study the general case where $\bo$ is crisp.

\begin{definition} Let  $\logic{L}_\bo^c$ be the logic obtained from $\logic{L}_\bo$ by adding the following axiom:

\begin{itemize}

\item[(c)] $\bo\varphi \vee \neg\bo\varphi$

\end{itemize}

\noindent A $\logic{L}_\bo^c$-algebra $\bf A$ is a $\logic{L}_\bo$-algebra such that $\bo(x)  \lor \neg \bo(x) = 1$ for every $x \in A$.
\end{definition}

Since it  is an  axiomatic extension of the logic $\logic{L}_\bo$, it turns out that $\logic{L}_\bo^c$ is algebraizable, whose equivalent algebraic semantics is given by the quasi-variety of $\logic{L}_\bo^c$-algebras, and semilinear as well, and thus complete with respect the class of $\logic{L}_\bo^c$-chains. From the definition above, it is clear that the operator $\bo$ in any  $\logic{L}_\bo^c$-chain $A$ is such that $\bo(x) \in \{0, 1\}$ for every $x \in A$. Moreover, this implies that the set $\{ x \in A \setminus \{0\} \ : \ \bo(x) = 1\}$ is an interval containing $N(A) \cup \{1\}$. 

Let us consider now the logics corresponding to the minimal and maximal (pointwisely) consistency operators, as announced in the introduction of Section \ref{extensions}.
First, consider $\logic{L}^{min}_\bo$ to be the axiomatic extension of the logic $\logic{L}_\bo$ with the following axiom:

\begin{itemize}

\item[(A4)] $\varphi \lor \neg\varphi \lor \neg\bo \varphi$

\end{itemize}
Since it is an axiomatic extension of $\logic{L}_\bo$,  $\logic{L}^{min}_\bo$ is complete with respect to the class of $\logic{L}^{min}_\bo$-chains, i.e.\ $\logic{L}_\bo$-chains satisfying the equation $x \lor \neg x \lor \neg\bo(x) = 1$. One can readily check that the equation $x \lor \neg x \lor \neg \bo(x) = 1$ holds in an  $\logic{L}_\bo$-chain only in the case that $\bo(x) = 0$ when $0 < x < 1$. Indeed, if $\min(x, \neg x) > 0$ it is clear that $\bo(x)$ has to be $0$, while if $x < 1$ and $\neg x  = 0$ then (A4) forces $\neg \bo(x) = 1$, that is $\bo(x) = 0$. Therefore, the $\bo$-operator in any $\logic{L}^{min}_\bo$-chain  is completely determined, and it is indeed the (pointwisely) minimal one definable in a $\logic{L}_\bo$-chain.

\begin{proposition} The logic $\logic{L}^{min}_\bo$ is  complete with respect to the class of $\logic{L}^{min}_\bo$-chains, i.e.\ $\logic{L}_\bo$-chains where the $\bo$ operator is the minimal one.
\end{proposition}

Since  $\logic{L}^{min}_\bo$ are a special kind of $\logic{L}^c_\bo$-chains, this proposition yields that  $\logic{L}^{min}_\bo$ must be an axiomatic extension of $\logic{L}^c_\bo$. Moreover,  it turns out that, for all $x$ in a  $\logic{L}^{min}_\bo$-chain,  $\bo(x)$ coincides with $\triangle(x \vee \neg x)$, where $\triangle$ is the Baaz-Monteiro projection operator, as it happened in the case of  $\logic{L}^{\neg\neg}_\bo$-chains.
Using the fact that both logics $\logic{L}^{min}_\bo$ and $\logic{L}_\triangle$ are chain-complete, it follows that they are inter-definable.

\begin{proposition} The logics $\logic{L}^{min}_\bo$ and $\logic{L}_\triangle$ are inter-definable  by means of the following translations:

\begin{itemize}
\item[(i)] from $\logic{L}^{min}_\bo$ to $\logic{L}_\triangle$: define $\triangle \varphi$ as $\varphi \land \bo\varphi$

\item[(ii)] from $\logic{L}_\triangle$  to $\logic{L}^{min}_\bo$: define $\bo \varphi$ as $\triangle(\varphi \lor \neg \varphi)$.
\end{itemize}

\end{proposition}

By (ii) of the above proposition $\bo\varphi$ is equivalent to the formula $\triangle(\varphi \lor \neg \varphi)$ in $\logic{L}_\triangle$. Thus by axiom ($\triangle 1$) the following result (proving the axiom of $\logic{L}^c_\bo$) is obvious.

\begin{lemma}   $\logic{L}^{min}_\bo$ proves the axiom $(c)$, i.e. $\bo\varphi \lor \neg \bo \varphi$.
\end{lemma}

Finally, consider the logic $\logic{L}^{max}_\bo$ to be the  extension of the logic $\logic{L}_\bo$ with the following inference rule:

\begin{itemize}
\item[] {\rm($\neg\neg_\bo$)}  \, $\displaystyle \frac{ \neg \neg \varphi \lor \delta}{ \bo\varphi \lor \delta}$
\end{itemize}

Again, since $(\neg\neg_\bo)^{\lor}$ is closed under disjunction, $\logic{L}^{max}_\bo$ is complete with respect to $\logic{L}^{max}_\bo$-chains, i.e. $\logic{L}_\bo$-chains where the following condition holds: if $\neg x = 0$ then
$\bo(x) = 1$. Since $\bo(x) = 0$ for all $x > 0$ such that $\neg x > 0$, then it is clear that $\bo$ is completely determined in such a chain and defined as: $\bo(x ) = 0$ if $0 < x \land \neg x$ and $\bo(x) = 1$ otherwise (i.e.\ if $x \in\{ 0, 1\}$ or  $\neg x = 0$). Hence $\bo$ is the maximal (pointwisely) consistency operator definable in a $\logic{L}$-chain.

\begin{proposition} The logic $\logic{L}^{max}_\bo$ is  complete with respect to the class of  $\logic{L}^{max}_\bo$-chains, i.e.\ $\logic{L}_\bo$-chains where the $\bo$ operator is the maximal one.
\end{proposition}

As a final remark, we notice that in case $\logic{L}$ is an extension of the basic fuzzy logic $\logic{BL}$, the above rule $(\neg\neg_\bo)$ can be equivalently replaced by the following axiom:

\begin{itemize}
\item[]  $(\neg\neg \varphi \to \varphi) \lor \bo \varphi$
\end{itemize}
Indeed, it is not difficult to check that, given the special features of negations in BL-chains, a consistency operator  $\bo$  in a  BL-chain $\bf A$ satisfies this axiom iff $\bo(x) = 0$ for $x$ such that $0 < \min(x, \neg x) $ and $\bo(x) = 1$ otherwise. Therefore, the quasivariety of $\logic{L}^{max}_\bo$-algebras is in fact a variety when $\logic{L}$ is a BL-extension, but whether the class of $\logic{L}^{max}_\bo$-algebras is a variety in a more general case remains as an open problem.

Figure \ref{summary} gathers the axiomatizations (relative to $\logic{L}$) of the logic $\logic{L}_\bo$ and of the different extensions  we have defined in Section \ref{extensions}.

\begin{figure}[htbp]
\begin{center}

\begin{tabular}{| c | l | c |}
\hline
Logic &  \hspace*{2.5cm} Definition & Operator $\bo$ \\
\hline
$\logic{L}_\bo$ &
\begin{tabular}{lll}
$\logic{L} \mbox{ } + $ & (A1) & $\neg(\varphi \land \neg \varphi \land \bo \varphi)$ \\
& (A2) &  $\bo \overline{1}$ \\
& (A3) & $ \bo \overline{0}$ \vspace{0.2cm}\\
& (Cong) &  { $ \displaystyle \frac{(\varphi \leftrightarrow \psi) \lor \delta}{(\bo \varphi \leftrightarrow \bo \psi) \lor \delta}$}  \vspace{0.3cm} \\
& (Coh) & { $ \displaystyle \frac{(\neg\neg \varphi \land (\varphi \to \psi)) \lor \delta}{(\bo \varphi \to \bo \psi) \lor \delta}$}\\
\end{tabular}
&
\begin{tabular}{c}
\\
\includegraphics[width=3.5cm]{figura-bola1.pdf}
\end{tabular}
\\
\hline
$\logic{L}^{\neg\neg}_\bo$ &
\begin{tabular}{lll}
$\logic{L}_\bo \mbox{ } + $ & $(\neg\neg)$ &  { $ \displaystyle \frac{\neg\neg\varphi}{\varphi}$} \\
\end{tabular}
&
\begin{tabular}{c}
\\
\includegraphics[width=3.5cm]{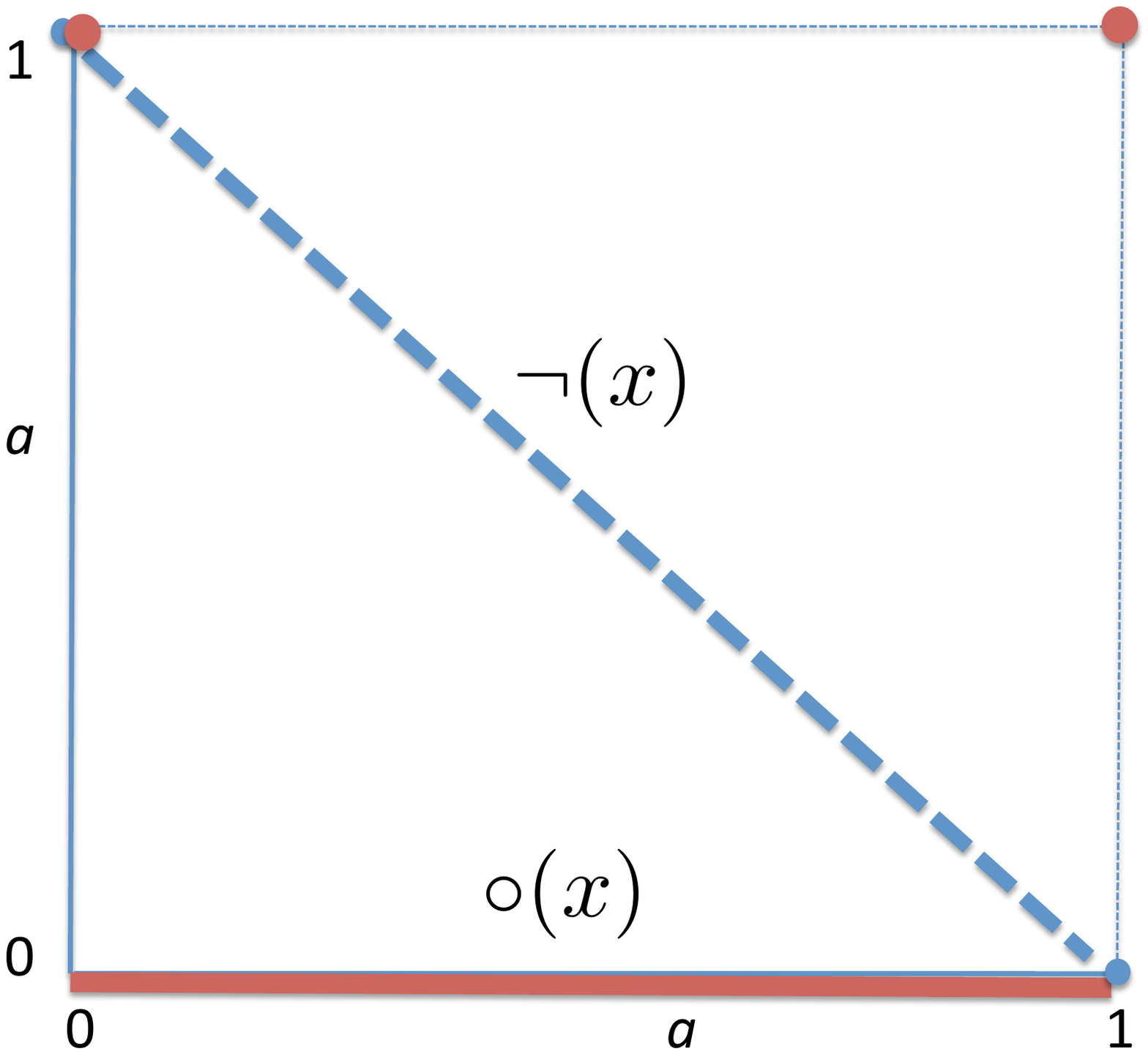}
\end{tabular}
\\
\hline
$\logic{L}^{c}_\bo$ &
\begin{tabular}{lll}
$\logic{L}_\bo \mbox{ } + $ & $(c)$ &  $\bo \varphi \lor \neg \bo \varphi$ \\
\end{tabular}
&
\begin{tabular}{c}
\\
\includegraphics[width=3.5cm]{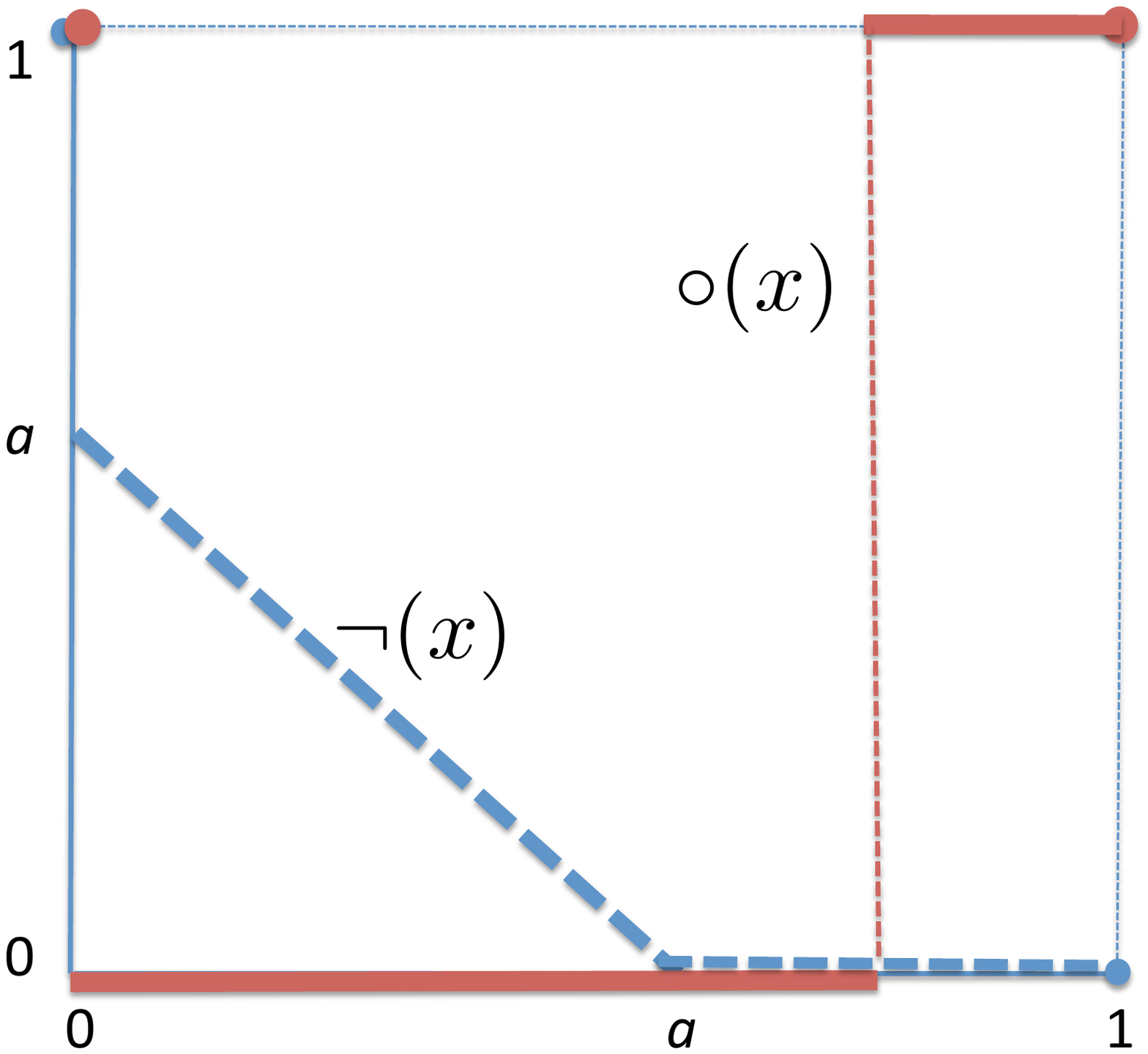}
\end{tabular}
\\
\hline
$\logic{L}^{min}_\bo$ &
\begin{tabular}{lll}
$\logic{L}_\bo \mbox{ } + $ & (A4) &  $ \varphi \lor \neg \varphi \lor \neg \bo \varphi$ \\
\end{tabular}
&
\begin{tabular}{c}
\\
\includegraphics[width=3.5cm]{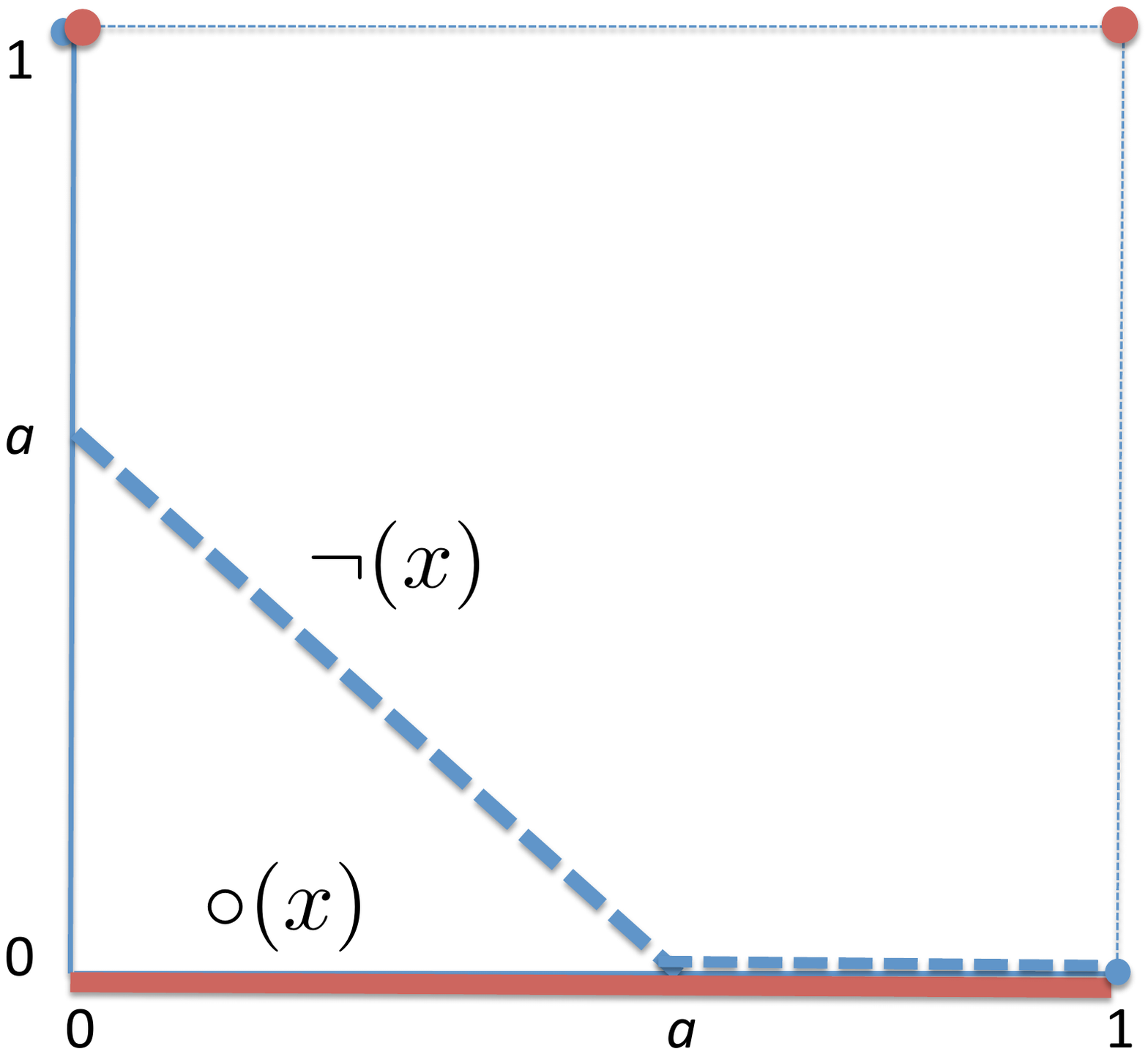}
\end{tabular}
\\
\hline
$\logic{L}^{max}_\bo$ &
\begin{tabular}{lll}
$\logic{L}_\bo \mbox{ } + $ & ($\neg\neg_\bo$)&   $\displaystyle \frac{ \neg \neg \varphi \lor \delta}{ \bo\varphi \lor \delta}$    \\
\end{tabular}
&
\begin{tabular}{c}
\\
\includegraphics[width=3.5cm]{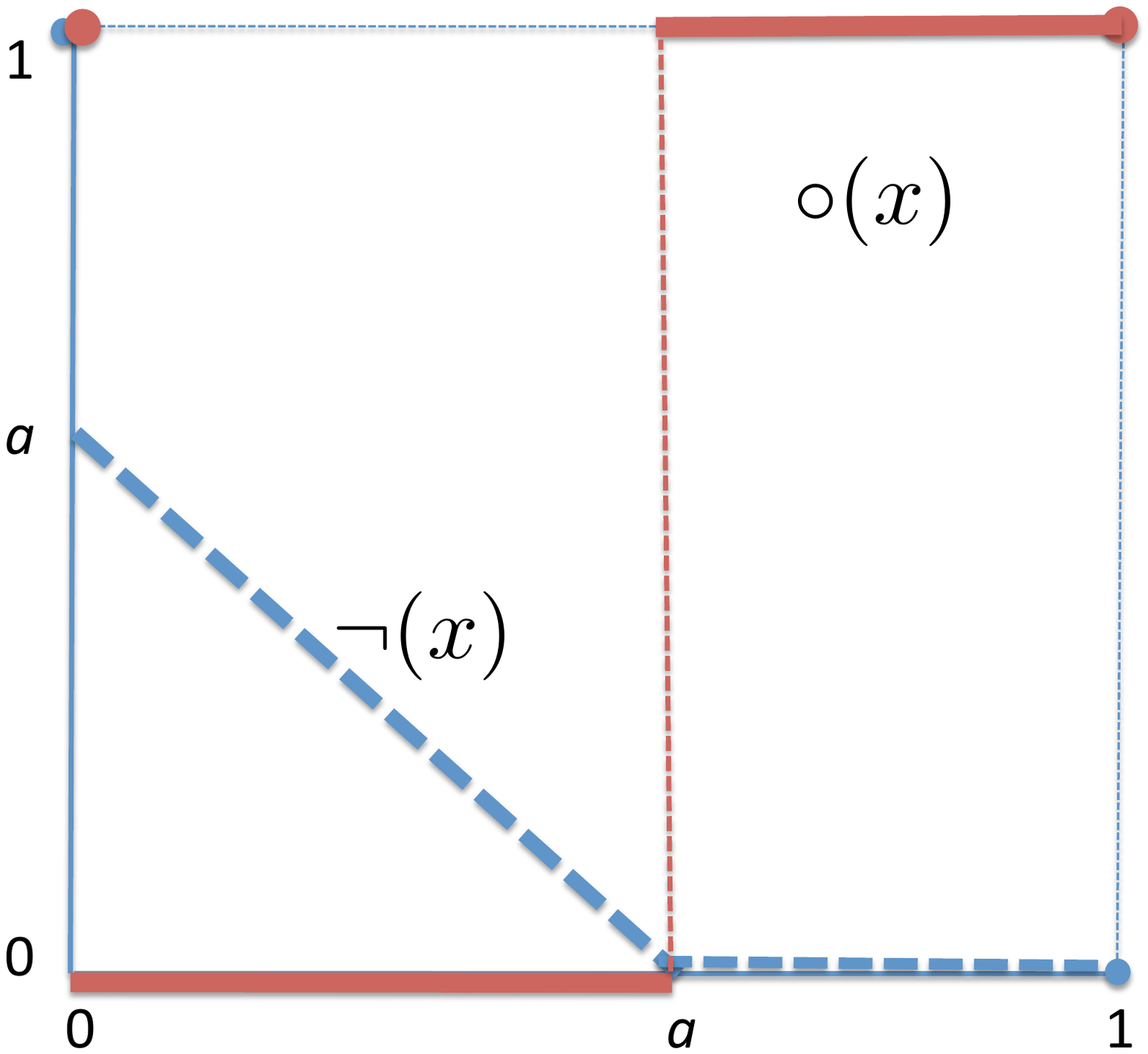}
\end{tabular}
\\
\hline
\end{tabular}

\caption{Summary of the axiomatizations of the logics $\logic{L}_\bo$, $\logic{L}^{\neg\neg}_\bo$, $\logic{L}^c_\bo$, $\logic{L}^{min}_\bo$, $\logic{L}^{max}_\bo$, and corresponding  graphs of  the operators $\bo$ on a standard BL-chain, where $ 0 < a \leq 1$.} 
\label{summary}
\end{center}
\end{figure}

\begin{remark} \label{remark-cons} Taking into account the graphs of the $\bo$ operators on real chains associated to  the logics $\logic{L}^{\neg\neg}_\bo$, $\logic{L}^c_\bo$, $\logic{L}^{min}_\bo$ and  $\logic{L}^{max}_\bo$ (see Figure \ref{summary}), and the proof of conservativeness of $\logic{L}_\bo$ with respect to $\logic{L}$ in Proposition \ref{p:ConsExp}, it is clear that the same kind of proof also applies to all these extensions of  $\logic{L}_\bo$. Hence  $\logic{L}^{\neg\neg}_\bo$ is a conservative expansion of  $\logic{L}^{\neg\neg}$ and $\logic{L}^c_\bo$, $\logic{L}^{min}_\bo$ and  $\logic{L}^{max}_\bo$ are conservative expansions of $\logic{L}$. 
\end{remark}

\subsection{Axiomatizing paraconsistent fuzzy logics $\logic{L}^{\mbox{\tiny $\leq$ }}$ with $\bo$}

As mentioned in the introduction of this section we know that the only paraconsistent fuzzy logic are the logics $\logic{L}^{\mbox{\tiny $\leq$ }}$ when $\logic{L}$ is not an SMTL logic. Thus our ultimate goal is the axiomatization of the expansion of paraconsistent logics $\logic{L}^{\mbox{\tiny $\leq$ }} $with a consistency operator $\bo$, that will be denoted $\logic{L}^{\mbox{\tiny $\leq$}}_\bo $. But from results of this section we know how to axiomatize the logics $L_\bo$ and, as shown in Section \ref{prelim-1} we know how to get an axiomatization of  $\logic{L}^{\mbox{\tiny $\leq$}}_\bo $ from the one of $\logic{L}_\bo $. Indeed the axiomatization of $\logic{L}^{\mbox{\tiny $\leq$}}_\bo $ is obtained by taking the same axioms of  $\logic{L}_\bo$ and adding the following inference rules: \vspace{0.2cm}

\begin{description}

\item[(Adj-$\wedge$)] from $\f$ and $\p$ deduce $\f\wedge\p$

\item[(MP-$r$)] if \mbox{}  $\vdash_{\logic{L_{\bo}}}\f\to\p$ \mbox{} (i.e. if  $\f\to\p$ is a  theorem of $\logic{L}_\bo$), then from $\f$   derive $\p$

\item [(Cong-$r$)] if \mbox{ } $\vdash_{\logic{L_{\bo}}}(\varphi \leftrightarrow \psi) \vee \delta$ \mbox{ } then derive \mbox{ } $(\bo\varphi \leftrightarrow \bo\psi) \vee \delta$

\item[(Coh-$r$)] if \mbox{ } $\vdash_{\logic{L_{\bo}}} (\neg\neg\varphi \wedge (\varphi\to \psi)) \vee \delta$ \mbox{ } then derive \mbox{ } $ (\bo\varphi \to \bo\psi) \vee \delta$

\end{description}

In the same way we could obtain axiomatizations of the logics $\logic{L}^{\mbox{\tiny $\leq$}}_\bo $ when $\logic{L}$ is any of the logics studied in this section. Then axiomatizations of the logics $(\logic{L}_\bo ^{\neg \neg})^{\mbox{\tiny $\leq$}}$, $(\logic{L}_\bo^c)^{\mbox{\tiny $\leq$}}$, $(\logic{L}^{min}_\bo)^{\mbox{\tiny $\leq$}}$ and $(\logic{L}^{max}_\bo)^{\mbox{\tiny $\leq$}}$ are easily obtained. Figure \ref{summary-2} gathers the inference rules of these logics (recall that the axioms coincide with those of the corresponding truth-preserving logics). Therefore we have defined and axiomatized a general family of LFIs  based on fuzzy logics and some of its extensions.

\begin{figure}[htbp]
\begin{center}
\begin{tabular}{| c | l |}
\hline
Logic & \hspace*{3cm}Inference rules   \\
\hline
$\logic{L}^{\mbox{\tiny $\leq$}}_\bo$  &
\begin{tabular}{lll} \\
rules of $\logic{L}^{\mbox{\tiny $\leq$}}$ \mbox{ } + \mbox{}  &  (Cong-r) &  $ \displaystyle \frac{\vdash_{\logic{L}_\bo} (\varphi \leftrightarrow \psi) \lor \delta}{ (\bo \varphi \leftrightarrow \bo \psi) \lor \delta}$ \vspace{0.3cm}  \\
& (Coh-r) & $ \displaystyle \frac{\vdash_{\logic{L}_\bo} (\neg\neg \varphi \land (\varphi \to \psi)) \lor \delta}{(\bo \varphi \to \bo \psi) \lor \delta}$\\ \\
\end{tabular}
\\
\hline
$(\logic{L}_\bo ^{\neg \neg})^{\mbox{\tiny $\leq$}}$ &
\begin{tabular}{lll} \\
rules of $\logic{L}_\bo^{\mbox{\tiny $\leq$}}$ \mbox{ } + \mbox{}  & $(\neg\neg$-$r)$ &  { $ \displaystyle \frac{\vdash_{\logic{L}^{\neg\neg}_\bo} \neg\neg\varphi}{\varphi}$} \\ \\
\end{tabular}
\\
\hline
$(\logic{L}_\bo^c)^{\mbox{\tiny $\leq$}}$ &
\begin{tabular}{lll} \\
rules of $\logic{L}_\bo^{\mbox{\tiny $\leq$}}$   &  \\ \\
\end{tabular}
\\
\hline
$(\logic{L}^{min}_\bo)^{\mbox{\tiny $\leq$}}$ &
\begin{tabular}{lll} \\
rules of $\logic{L}_\bo^{\mbox{\tiny $\leq$}}$  &  \\  \\
\end{tabular}
\\
\hline
$(\logic{L}^{max}_\bo)^{\mbox{\tiny $\leq$}}$ &
\begin{tabular}{lll} \\
rules of $\logic{L}_\bo^{\mbox{\tiny $\leq$}}$ \mbox{ } + \mbox{} & $(\neg\neg_\bo$-$r)$ & $\displaystyle \frac{ \vdash_{\logic{L}^{max}_\bo} \neg \neg \varphi \lor \delta}{ \bo\varphi \lor \delta}$  \\ \\
\end{tabular}
\\
\hline
\end{tabular}

\caption{Summary of the inference rules of the logics $\logic{L}^{\mbox{\tiny $\leq$}}_\bo$, $(\logic{L}_\bo ^{\neg \neg})^{\mbox{\tiny $\leq$}}$, $(\logic{L}_\bo^c)^{\mbox{\tiny $\leq$}}$, $(\logic{L}^{\min}_\bo)^{\mbox{\tiny $\leq$}}$ and $(\logic{L}^{\max}_\bo)^{\mbox{\tiny $\leq$}}$
.}
\label{summary-2}
\end{center}
\end{figure}

Finally, recall that, as observed at the end of Section \ref{prelim-1}, if a logic $\logic{L}'$ is a conservative expansion of another $\logic{L}$, then $\logic{L'}^{\mbox{\tiny $\leq$ }}$ is also a conservative expansion of $\logic{L}^{\mbox{\tiny $\leq$ }}$. Therefore, taking into account Proposition \ref{p:ConsExp} and Remark \ref{remark-cons}, we get the following immediate result. 

\begin{proposition}  The logics $\logic{L}^{\mbox{\tiny $\leq$}}_\bo $, $(\logic{L}_\bo^c)^{\mbox{\tiny $\leq$}}$, $(\logic{L}^{min}_\bo)^{\mbox{\tiny $\leq$}}$ and $(\logic{L}^{max}_\bo)^{\mbox{\tiny $\leq$}}$ are conservative expansions of $\logic{L}^{\mbox{\tiny $\leq$}}$, while  $(\logic{L}_\bo ^{\neg \neg})^{\mbox{\tiny $\leq$}}$ is a conservative expansion of $(\logic{L} ^{\neg \neg})^{\mbox{\tiny $\leq$}}$. 
\end{proposition}

\section{About the propagation property} \label{propag}

One of the distinctive features of da Costa's C-systems is the so-called {\em propagation property} of the consistency connective $\bo$, which states that consistency (or well-behavior, according to da Costa's terminology) is propagated in the following sense: from $\bo\varphi$ it follows $\bo\neg\varphi$, and from $\{\bo\varphi,\bo\psi\}$ it follows $\bo(\varphi\#\psi)$, for every binary connective $\#$. We can adapt this property to our setting, and study  conditions which ensure its validity.

\begin{definition}  Let $\logic{L}$ be a paraconsistent fuzzy logic with a consistency operator $\bo$. Then we say that $\bo$ satisfies the \emph{propagation property} in $\logic{L}$ with respect to a subset $X$ of connectives of the language of $\logic{L}$  if
$$\bo\varphi_1,\ldots, \bo\varphi_n  \vdash_{\logic{L}}\bo\#(\varphi_1,\ldots,\varphi_n), $$ for every n-nary connective $\# \in X$ and formulas  $\varphi_1,\ldots,\varphi_n$ built with connectives from $X$.
\end{definition}

Observe that in the case $n = 0$, $\#$ is a constant and that the above condition requires $\vdash_L \bo \#$.

The paraconsistent fuzzy logics studied in this paper are  logics  $\logic{L}_\bo^{\mbox{\tiny $\leq$ }}$ and some extensions, where $\logic{L}$ is a truth-preserving fuzzy logic.  Knowing the relation between truth-preserving and degree-preserving fuzzy logics, $\bo$ satisfies the \emph{propagation property} in a logic $\logic{L}_\bo^{\mbox{\tiny $\leq$ }}$ with respect to some subset of connectives $X \subseteq \{ \0, \land, \&, \to\}$\footnote{We are  assuming here to work with a core fuzzy logic, and in core fuzzy logics these are the primitive connectives, the rest are definable from them.} whenever: 
$$ \mbox{ } \hspace*{1cm}  \left \{ \begin{array}{ll}
\vdash_{\logic{L}_\bo} \bo\0 & \mbox{if } \0 \in X \\
\vdash_{\logic{L}_\bo} (\bo\varphi \wedge \bo\psi) \to \bo(\varphi \# \psi), & \mbox{for each binary } \#  \in X \\
\end{array}
\right .  \hspace*{0.8cm} \hfill (Prop*)
$$
In such a case we will say that $\bo$ satisfies the propagation property (Prop*) in the logic  $\logic{L}_\bo^{\mbox{\tiny $\leq$ }}$ with respect to the set of connectives $X$. Actually since in the logics $\logic{L}_\bo$ we have $\vdash_L \bo \0$, the first condition is always satisfied and it can be dropped from (Prop*).

\begin{proposition} \label{prop} The following conditions hold:
\begin{enumerate}
\item
 $\bo$ satisfies (Prop*) in any fuzzy logic $\logic{L}_\bo$ with respect to the set of primitive connectives $\{\land,  \to\}$.

\item $\bo$ satisfies (Prop*) in any fuzzy logic of the families  $\logic{L}^{\neg\neg}_\bo$,  
$\logic{L}^{\min}_\bo$ or $\logic{L}^{\max}_\bo$ with respect to the set of primitive connectives $\{ \wedge, \&, \to\}$.\\

\end{enumerate}
\end{proposition}
\begin{proof}
Due to chain completeness of the logics involved, the whole proof is done by algebraic means. Let $\bf A$ be a  $\logic{L}_\bo$-chain. The first item for the connective $\wedge$ is a  consequence of the non-decreasing property of $\bo$ in $A \setminus \{0\}$ combined with  the fact that $\bo(0) = 1$. Before dealing with the property for $\to$, we first consider the property for the negation $\neg$, that is proved by cases. If $x \in \{0,1\}$ the property is obvious. If $x \in N({\bf A})$ then $\neg x = 0$, and thus $\bo(\neg x) = 1$ and the property obviously holds. Otherwise, if $x$ does not belong to the previous cases $x \wedge \neg x > 0$ then $\bo(x) = 0$ and also the property is obviously satisfied. Finally, for $\to$ the proof is easy since, remembering that a residuated implication satisfies the inequality $y \leq x \to y $, assuming $y > 0$ we have $\bo(x) \wedge \bo(y) \leq \bo(y) \leq \bo(x \to y)$; if $y = 0$ then we are back to the case of the negation. 

For the second item 
we only need to deal with the case of $\&$. First observe that if $\neg x = 0$ and  $\neg y = 0$, then $\neg(x \& y) = 0$ as well, since  $\neg(x \& y) = (x \& y) \to 0 = x \to (y \to 0) = x \to \neg y = x \to 0 = \neg x = 0$. Second observe that in the chains of the considered logics, the image of $\bo$ is $ \{0,1\}$. From there the proof is easy. If one of the values $x,y$ is 0 or 1 the result is obvious. If one of the values $x,y$ is in $(0,1) \setminus N(A)$ then $\bo(x) \wedge \bo(y) = 0$, and the implication is trivially valid. Finally if $x,y \in N(A)$ the result follows from the fact that $\bo$ is non-decreasing taking into account the first observation above.
\end{proof}

The first item of Proposition \ref{prop} can not be improved in the sense that $\logic{L}_\bo$ does not prove $\bo(\varphi) \wedge \bo(\psi) \to \bo (\varphi \& \psi)$,
 as the following example shows.

\begin{example}
Let $\logic{L}$ be the logic of the t-norm $\otimes$ that is the ordinal sum of a {\L}ukasiewicz component and a product component, with $\frac{1}{2}$ being the idempotent separating point. Then consider the $\logic{L}_\bo$-chain where the consistency operator $\bo$ is defined by $\bo(x) = 0$ if $x \in (0,\frac{3}{4})$ and $\bo(x) = 1$ otherwise. Take now $x= \frac{5}{6}$ and $y =  \frac{3}{4}$.  Then $\frac{5}{6} \otimes \frac{3}{4} < \frac{3}{4}$, and clearly $\bo(x) = \bo(y) = 1$ while $\bo( \frac{5}{6} \otimes \frac{3}{4}) = 0$.
\end{example}

\section{Recovering Classical Logic} \label{recov}

In the context of LFIs, it is a desirable property to recover the classical reasoning by means of
the consistency connective $\bo$ (see~\cite{car:con:mar:07}). Specifically, let {\bf CPL} be
classical propositional logic. If $\logic{L}$ is a given LFI such that its reduct to the language of
{\bf CPL} is a sublogic of  {\bf CPL}, then a DAT (Derivability Adjustment  Theorem) for $\logic{L}$
with respect to {\bf CPL} is as follows: for every finite set of formulas $\Gamma\cup\{\varphi\}$ in
the language of {\bf CPL}, there exists a finite set of formulas $\Theta$ in the language of
$\logic{L}$, whose variables occur in formulas of $\Gamma\cup\{\varphi\}$, such that
$$\mbox{(DAT)} \ \ \ \Gamma \vdash_{\bf CPL}\varphi \ \mbox{ iff } \ \bo(\Theta),\Gamma
\vdash_{\logic{L}}\varphi.$$
When the operator $\bo$ enjoys the propagation property in the logic $\logic{L}$  with respect to
the classical connectives (see the previous section) then the DAT takes the following, simplified
form: for every finite set of formulas $\Gamma\cup\{\varphi\}$ in the language of {\bf CPL},
$$\mbox{(PDAT)} \ \ \ \Gamma \vdash_{\bf CPL}\varphi \ \mbox{ iff } \ \bo p_1,\ldots,\bo p_n,\Gamma
\vdash_{\logic{L}}\varphi$$
where $\{ p_1,\ldots, p_n\}$ is the set of propositional variables occurring in
$\Gamma\cup\{\varphi\}$.

Here we are interested in investigating whether we can expect some form of the  (PDAT) for the logics $\logic{L}_\bo^{\mbox{\tiny $\leq$ }}$ and, as we have noted before, $\bo$ has the propagation property iff certain formulas are theorems in $\logic{L}_\bo$. Thus in terms of theoremhood, to have a DAT when $\bo$ propagates in $\logic{L}_\bo^{\mbox{\tiny $\leq$ }}$ is equivalent to prove the following:
$$\mbox{(PDAT$^*$)} \ \ \ \vdash_{\bf CPL}\varphi \ \mbox{ iff } \  \vdash_{\logic{L}_\bo}
\left(\bigwedge_{i=1}^n \bo p_i\right) \to\varphi   \ \mbox{ iff } \   \bigwedge_{i=1}^n \bo p_i \vdash_{\logic{L}^{\leq}_\bo} \varphi  $$
where $\{ p_1,\ldots, p_n\}$ is the set of propositional variables occurring in $\varphi$
(obviously, when this set is empty, that is, when $n=0$, then $\bigwedge_{i=1}^n \bo p_i$ is set to
be $\bar 1$). Arguably, (PDAT) (or (PDAT$^*$)) is more interesting than (DAT). For instance, the hierarchy $(C_n)_{n\geq 1}$ of paraconsistent logics  introduced by da Costa satisfies (PDAT).

Since $\bo$ satisfies the propagation property in logics $\logic{L}_\bo$ with respect to the classical
signature (cf. Proposition \ref{prop}), we  try to prove (PDAT*) for them. However, in the general setting of logics $\logic{L}_\bo$ (PDAT*) does not always hold. Indeed it is easy to see that $\vdash_{\bf CPL} p \vee \neg p$ but $\nvdash_{\logic{L}_\bo} \bo p\to (p \vee \neg p)$, i.e.  $\bo p\to (p \vee \neg p)$ is not a tautology over all $\logic{L}_\bo$-chains. Take for example the $\logic{L}_\bo$-chain defined on a $\logic{L}$-chain $\alg{A}$ by defining $\bo$ as follows: $\bo(x) =1$ if $x \in \{0,1\} \cup N(\alg{A})$ and $\bo(x) = 0$ otherwise. Then it is easy to see that if $N(\alg{A}) \neq \emptyset$ then $e(\bo p\to (p \vee \neg p)) \neq 1$ for any evaluation $e$ on $\bf A$ such that
 $e(p) \in N(\alg{A})$.

This example is significative since  the principle $\varphi \vee \neg \varphi$ is enough to collapse MTL-logic with classical logic. In fact we propose the following definition.

\begin{definition} Let  $\logic{L}_\bo^{dat}$ be the logic obtained from $\logic{L}_\bo$ by adding the following axiom:\\

{\rm($\bo$EM)} $\bo\varphi \to (\varphi \vee\neg\varphi)$

\end{definition}

By the same argument as above, $\logic{L}_\bo^{dat}$ is algebraizable and its algebraic semantics is given by the class of $\logic{L}_\bo^{dat}$-algebras.

\begin{definition} A $\logic{L}_\bo^{dat}$-algebra ${\bf A}$ is a $\logic{L}_\bo$-algebra  such that $\bo(x) \leq x \vee \neg x$ for every $x \in A$.
\end{definition}

Therefore $\logic{L}_\bo^{dat}$ is complete with respect to the variety of $\logic{L}_\bo^{dat}$-algebras and, more important, with respect to the chains of the variety (since the logic is an axiomatic extension and thus it is semilinear as $\logic{L}_\bo$).

Moreover since $\logic{L}_\bo^{dat}$ extends $\logic{L}_\bo$, it follows that $\bo$ satisfies the propagation property in $\logic{L}_\bo^{dat}$ with respect to the classical signature. 

However, since $\logic{L}_\bo^{dat}$ does not satisfy contraction, property (PDAT*)  will be hardly satisfied: it should be intuitively clear that, in some situations, it could be necessary to use the `consistency assumption'  (and so the law of excluded middle)  more than once in order to obtain a given tautology. We  show next that a slightly modified form of DAT indeed holds for $\logic{L}_\bo^{dat}$. 

\begin{proposition} \label{pdat} The logic $\logic{L}_\bo^{dat}$ satisfies the following form of DAT:
$$\mbox{(PDAT$^{**}$)} \ \ \ \vdash_{\bf CPL}\varphi \ \mbox{ iff  there is  $k\geq 1$ such that } \
\vdash_{\logic{L}_\bo^{dat}}\left(\bigwedge_{i=1}^n \bo p_i\right)^k \to\varphi$$
where $\{ p_1,\ldots, p_n\}$ is the set of propositional variables occurring in $\varphi$ and $\psi^k$ is as  a shorthand for $\psi \conj \stackrel{k}{\ldots} \conj \psi$. 
\end{proposition}
\begin{proof}
Let $\varphi$ be a formula in the language of {\bf CPL} and suppose that $p_1, p_2,\ldots,p_n$ are the propositional variables appearing in $\varphi$ . If $\vdash_{\bf CPL}\varphi$ then $\{p_i \vee \neg p_i  \ : \ i = 1,2,\ldots,n\} \vdash_{\logic{L}} \varphi$, since for any evaluation $e$ in any $\logic{L}$-chain, $e(p_i \vee \neg p_i) = 1$ iff $e(p_i)$ is either $0$ or $1$. Then by the local deduction-detachment theorem of $\logic{L}$, there is a natural $k$ such that   $\vdash_{\logic{L}}(\bigwedge_{i=1}^n (p_i \vee \neg p_i))^k \to\varphi$,
and this theorem is also valid in $\logic{L}_\bo^{dat}$. Then, by axiom ($\bo$EM), this implies  $\vdash_{\logic{L}_\bo^{dat}}(\bigwedge_{i=1}^n \bo p_i)^k \to\varphi$,  and hence (PDAT$^{**}$) holds.

Conversely, assume that $\vdash_{\logic{L}_\bo^{dat}}(\bigwedge_{i=1}^n \bo p_i)^k \to \varphi$ for some $k\geq 1$, and let $e$ be any evaluation on the 2-element Boolean algebra ${\bf B}_2$. Since ${\bf B}_2$ can be considered as a $\logic{L}_\bo^{dat}$-chain where $\bo(0)=\bo(1)=1$, then we have $e((\bigwedge_{i=1}^n \bo p_i)^k \to \varphi)=1$. But then we necessarily have $e(\varphi)=1$, because $e(\bigwedge_{i=1}^n \bo p_i)=1$. Therefore $\varphi$ is a {\bf CPL}-tautology and so $\vdash_{\bf CPL}\varphi$.
\end{proof}

An easy reasoning shows an analogous result when we have an arbitrary set of premises built from a finite set of propositional variables:
$$\Gamma \vdash_{\bf CPL}\varphi \ \mbox{ iff  there is  $k\geq 1$ such that } \
\Gamma \vdash_{\logic{L}_\bo^{dat}}\left(\bigwedge_{i=1}^m \bo p_i\right)^k \to \varphi$$
where now $p_1, \ldots ,p_m$ are the propositional variables appearing in $\Gamma \cup \{\varphi\}$.

One interesting question is to know examples where {\rm (PDAT$^{**}$)} is only true for some $k \geq 2$. This is an open question in general but we have the following illustrative example.
Consider the formula $\varphi = (p \lor \neg p)^k$ with $k > 1$. Since in a logic $\logic{L}_\bo$ both the weak and strong conjunctions $\land$ and $\&$  are many-valued generalizations of the classical conjunction, it is clear that $ (p \lor \neg p)^k$ is equivalent to the classical tautology $p \lor \neg p$ when we restrict to Boolean evaluations. However there are extensions of  $\logic{BL}$ logic  where($\bo$EM) is a tautology but $\bo p \to  (p \lor \neg p)^k$ is not. For instance let $\logic{L}$ be the logic of the $\logic{BL}$-chain $[0, 1]_{{\textrm{\L}}\oplus{\textrm{\L}}}$ where ${\textrm{\L}}\oplus{\textrm{\L}}$ is the ordinal sum of two copies of {\L}ukasiewicz standard algebra $[0,1]_{\textrm{\L}}$, and let $a$ be the idempotent element separating the two components. Further, take the operator $\bo$ in $[0, 1]$ defined by  $\bo(0) = \bo(1) = 1$, $\bo(x) = x$ if $x\in [a,1]$, and $\bo(x) = 0$ otherwise. An easy computation shows that $\bo p \to  (p \lor \neg p)^2$ is already not a tautology in $\logic{L}^{dat}_\bo$. Therefore, if we use $\&$ as conjunction symbol in the language of {\bf CPL}, {\rm (PDAT$^{*}$)} is not valid in $\logic{L}_\bo^{dat}$, while obviously {\rm (PDAT$^{**}$)} is so. 
Nevertheless we have not been able to find a similar example when we use $\land$ as conjunction symbol the language of  {\bf CPL}.

Finally, notice that axiom ($\bo$EM) is a theorem of the logics $\logic{L}^{\neg\neg}_\bo$ and $\logic{L}^{\min}_\bo$. Therefore, in these logics $\bo p_i \lor \neg \bo p_i$ is a theorem, hence it is clear that $\bo p_i$  is equivalent to $(\bo p_i)^k$ for any $k$, and thus we have the following direct corollary.

\begin{corollary}
Let $\Gamma \cup \{\varphi\}$ be a finite set of formulas in the language of {\bf CPL} and let $\{ p_1,\ldots, p_m\}$ the set of propositional variables appearing in $\Gamma \cup \{\varphi\}$. Then $$\Gamma \vdash_{\bf CPL}\varphi \ \mbox{ iff  } \
\Gamma \vdash_{\logic{L}_\bo^{+}}\left(\bigwedge_{i=1}^n \bo p_i\right) \to  \varphi$$
where $+ \in \{\neg \neg, {\rm min}\}$. 
\end{corollary}

\section{Inconsistency operators in the logic \MTL} \label{inconsist}

As recalled in Section~\ref{seclfi}, within the LFIs framework one can also consider an  {\em inconsistency operator} $\bon $, dual to the consistency operator $\bo$,  where $\bon\varphi$ has the intended meaning of $\neg\bo\varphi$ (see \cite{car:con:mar:07}).  

In this section we show how to add inconsistency operators to \MTL-algebras, as well as to its logical counterparts, based on the content of the previous sections in terms of consistency operators. 

\begin{definition}  Given a logic $\logic{L}$ that is a semilinear core expansion of \MTL\ but not \SMTL, we define the logic $\logic{L}_\bon$ as the expansion of $\logic{L}$ in a language which incorporates a new unary connective $\bon$ with  the following axioms: \vspace{0.2cm}

\begin{tabular}{ll}
{\rm(A1')} & $\neg(\varphi\wedge \neg\varphi )\vee \bon\varphi$\\
{\rm(A2')}  & $\neg\bon \bar 1$\\
{\rm(A3')} & $\neg\bon \bar 0$\\
\end{tabular}
\vspace{0.2cm}

\noindent
and the following inference rules:  \vspace{0.2cm}

\begin{tabular}{ll}
 {\rm(Cong')} \, $\displaystyle \frac{  (\varphi \leftrightarrow \psi) \vee \delta}{ (\bon\varphi \leftrightarrow \bon\psi) \vee \delta}$ \hspace{1cm} & {\rm(Coh')} \,$\displaystyle \frac{  (\neg\neg\varphi \wedge (\varphi\to \psi)) \vee \delta}{ (\bon\psi \to \bon \varphi) \vee \delta}$  \vspace{0.4cm} \\
\end{tabular}

\end{definition}

As in the case of $\logic{L}_\bo$, due to the presence of the rule  {\rm(Cong')},  $\logic{L}_\bon$ is a Rasiowa-implicative logic, and thus it  is also algebraizable in the sense of Blok and Pigozzi and its algebraic semantics is given by $\logic{L}_\bon$-algebras.

\begin{definition} A $\logic{L}_\bon$-algebra ${\bf A}$ is an expansion of a $\logic{L}$-algebra with a new unary operation $\bon: A \to A$ satisfying  the following conditions for all $x, y, z\in A$:
\begin{enumerate}
\item[($\bon 1$)] $ \neg(x \wedge \neg x)  \vee \bon(x) = 1$
\item[($\bon 2$)] $\bon(1) = \bon(0) = 0$
\item[($\bon 3$)] if  $(\neg\neg x \wedge (x \to y)) \vee z = 1$ then $(\bon(y) \to \bon(x)) \vee z = 1$.
\end{enumerate}
\end{definition}
Again, since the rules  {\rm(Cong')} and  {\rm(Coh')} are closed under $\lor$-forms, $\logic{L}_\bon$ is complete with respect to the class of $\logic{L}_\bon$-chains. Obviously, the $\bon$ operators in  $\logic{L}$-chains have the dual form of the $\bo$ operators (described in Figure \ref{neg-bola}), and we will not go into further details. 

The intended duality between both operators $\bo$ and $\bon$ is made explicit in the following results.

\begin{proposition} \label{translat}
Let $t$ be a translation map from the language of  $\logic{L}_\bo$ to the language of  $\logic{L}_\bon$ which replaces $\bo$ by $\neg\bon$. Conversely, let $t'$ be the translation map in the opposite direction, which replaces $\bon$ by $\neg\bo$. Then $\Gamma \vdash_{\logic{L}_\bo}\varphi$ implies that  $t(\Gamma) \vdash_{\logic{L}_\bon}t(\varphi)$ and $\Gamma' \vdash_{\logic{L}_\bon}\varphi'$ implies that  $t'(\Gamma') \vdash_{\logic{L}_\bo}t'(\varphi')$.
\end{proposition}
\begin{proof}

It is enough to prove that: (i) the translation of each axiom of the source logic can be derived in the target logic, and (ii) the translation of each inference rule  of the source logic is an inference rule which is derivable in the target logic. The proof easily follows by using that in MTL the following formulas are theorems: $\neg (\varphi \land \psi) \leftrightarrow (\neg \varphi \lor \neg \psi)$, $\varphi \to \neg\neg \varphi$ and $(\varphi \to \psi) \to (\neg \psi \to \neg \varphi)$. 
\end{proof}

Notice however that the above translations do not yield that both logics are equivalent: indeed, the translations $t$ and $t'$ are not in general  each other's inverse. This is due to the fact that, e.g. in $\logic{L}_\bo$, $\bo\varphi$ in general is not equivalent to $\neg \neg \bo\varphi$. However, in the frame of the logic $\logic{L}^c_\bo$ and its extensions, where $\bo\varphi$ is Boolean, one can prove $\bo\varphi \leftrightarrow \neg \neg \bo\varphi$, and one can establish their equivalence with  their $\bon$-dual corresponding logics. Notice that if  $\logic{L}$ is an IMTL logic (i.e. whose negation is involutive), then $\logic{L}^c_\bo$ coincides with $\logic{L}_\bo$ itself. 

In particular, if we define dual counterparts of the logics  $\logic{L}^{\neg\neg}_\bo$,  $\logic{L}^c_\bo$, $\logic{L}^{\min}_\bo$ or $\logic{L}^{\max}_\bo$ as:

\begin{itemize}
\item[-]  $\logic{L}^{\neg\neg}_\bon$: is the extension of  $\logic{L}_\bon$ with the rule ``from $\neg\neg \varphi$ infer $\varphi$''

\item[-]  $\logic{L}^c_\bon$: is the axiomatic extension of  $\logic{L}_\bon$  with  the axiom $\bon\varphi \lor \neg \bon \varphi$

\item[-]  $\logic{L}^{\max}_\bon$: is the axiomatic extension of  $\logic{L}_\bon$ with the axiom $\varphi \lor \neg\varphi \lor \bon\varphi$

\item[-]  $\logic{L}^{\min}_\bon$: is the  extension of  $\logic{L}_\bon$ with the rule ``from $\neg\neg \varphi \lor \delta$ infer $\neg\bon\varphi \lor \delta$''
\end{itemize}
then we can list the following equivalences (denoted by $\equiv$) among logics via the translations $t$ and $t'$:  $\logic{L}^{\neg\neg}_\bo \equiv \logic{L}^{\neg\neg}_\bon$, $\logic{L}^c_\bo \equiv \logic{L}^c_\bon$, $\logic{L}^{\min}_\bo \equiv \logic{L}^{\max}_\bon$ and $\logic{L}^{\max}_\bo \equiv \logic{L}^{\min}_\bon$.\footnote{Keep in mind that, since we have kept the superscripts $\min$ and $\max$ for the  logics that respectively correspond to minimum and maximum $\bon$ operators, the dual $\logic{L}^{\min}_\bo$ is $\logic{L}^{\max}_\bon$ and viceversa, the dual of $\logic{L}^{\max}_\bo$ is $\logic{L}^{\min}_\bon$.} As a consequence, the (quasi) varieties associated to pairs of equivalent logics are termwise equivalent. 

To conclude, just to point out that, as in the case of the consistency operators, the paraconsistent versions of the above logics with the inconsistency operators $\bon$ would correspond to their degree-preserving counterparts, namely the logics  $\logic{L}^{\mbox{\tiny $\leq$}}_\bon$, $(\logic{L}_\bon ^{\neg \neg})^{\mbox{\tiny $\leq$}}$, $(\logic{L}_\bon^c)^{\mbox{\tiny $\leq$}}$, $(\logic{L}^{\min}_\bon)^{\mbox{\tiny $\leq$}}$ and $(\logic{L}^{\max}_\bon)^{\mbox{\tiny $\leq$}}$.

\section{Concluding remarks} \label{conclu}

In this paper we have investigated the possibility of defining paraconsistent logics of formal inconsistency (LFIs) based on systems of mathematical fuzzy logic, in particular by first expanding axiomatic extensions of the fuzzy logic MTL with the characteristic consistency and inconsistency operators of LFIs, and then by considering their degree-preserving versions, that are paraconsistent. Actually, in the same line of \cite{er-es-fla-go-no:2013} and based on a novel perspective, this paper intends to contribute to the study and understanding of the relationships between paraconsistency and fuzziness. 

\subsection*{Acknowledgments} 
The authors have been partially supported by the FP7-PEOPLE-2009-IRSES project MaToMUVI (PIRSES-GA-2009-247584). Coniglio was also supported by FAPESP (Thematic Project LogCons 2010/51038-0), and by a research grant from CNPq (PQ 305237/2011-0). Esteva and Godo also acknowledge partial support by the MINECO project~\mbox{TIN2012-39348-C02-01}.

\bibliographystyle{plain}

\end{document}